\documentclass[twoside,11pt]{amsart}

\usepackage[utf8]{inputenc} 
\usepackage[T1]{fontenc}    
\usepackage{amsmath}
\usepackage{natbib}
\usepackage{hyperref}       
\hypersetup{linkcolor=blue}
\usepackage{graphicx}
\usepackage{url}            
\usepackage{booktabs}       
\usepackage{amsfonts}       
\usepackage{nicefrac}       
\usepackage{microtype}      
\usepackage{xcolor}         

\usepackage{enumerate}
\usepackage{enumitem}
\usepackage{csquotes}

\hypersetup{colorlinks=true, linkcolor=blue, citecolor=blue}

\usepackage{thm-restate}

\usepackage{amsmath}
\usepackage{amssymb}
\usepackage{amsthm}
\usepackage{mathtools}

\theoremstyle{plain}
\newtheorem{theorem}{Theorem}

\newtheorem{lemma}[theorem]{Lemma}

\theoremstyle{definition}

\newtheorem{assumption}{Assumption}

\theoremstyle{remark}
\newtheorem{remark}[theorem]{Remark}

\newcommand{\E}{\mathbb{E}}
\newcommand{\R}{\mathbb{R}}
\newcommand{\Prob}{\mathbb{P}}

\title[A.s. convergence of stochastic Hamiltonian descent methods]{Almost sure convergence of stochastic Hamiltonian descent methods}

\author[M. Williamson]{M\r{a}ns Williamson}
\address{Centre for Mathematical Sciences, Lund University, P.O.\ Box 118, 221 00 Lund, Sweden}
\email{mans.williamson@math.lth.se}

\author[T. Stillfjord]{Tony Stillfjord}
\address{Centre for Mathematical Sciences, Lund University, P.O.\ Box 118, 221 00 Lund, Sweden}
\email{tony.stillfjord@math.lth.se}

\thanks{This work was partially supported by the Wallenberg AI, Autono\-mous Systems and Software Program (WASP) funded by the Knut and Alice Wallenberg Foundation. The authors declare no competing interests.
}

\begin{document}

\begin{abstract}
Gradient normalization and soft clipping are two popular techniques
for tackling instability issues and improving convergence of
stochastic gradient descent (SGD) with momentum.  In this article, we
study these types of methods through the lens of dissipative
Hamiltonian systems. Gradient normalization and certain types of soft
clipping algorithms can be seen as (stochastic) implicit-explicit
Euler discretizations of dissipative Hamiltonian systems, where the
kinetic energy function determines the type of clipping that is
applied.  We make use of dynamical systems theory to show in a unified
way that all of these schemes converge to stationary points of the
objective function, almost surely, in several different settings: a)
for $L$-smooth objective functions, when the variance of the
stochastic gradients is possibly infinite, b) under the
$(L_0,L_1)$-smoothness assumption, for heavy-tailed noise with bounded
variance, and c) for $(L_0,L_1)$-smooth functions in the empirical risk
minimization setting, when the variance is possibly infinite but the
expectation is finite.
\end{abstract}

\maketitle

\section{Introduction}\label{sec:introduction}
In this article we consider the unconstrained optimization problem
\begin{align}\label{problem:objective_function}
\min_{q \in \mathbb{R}^d} F(q),
\end{align}
where $F: \mathbb{R}^d \to \mathbb{R}$ is an objective function. A common case in mathematical statistics and machine learning is the empirical risk minimization setting, where $F$ is a weighted sum of loss functions:
\begin{align}\label{eq:F_sum}
F(q) = \frac{1}{N} \sum_{i=1}^N  f_i(q),
\end{align}
with $f_i(q) = \ell(h(x_i,q),y_i)$.
Here, $\{(x_i,y_i)\}_{i=1}^N$ is an underlying data set of feature-label pairs in the feature-label space $\mathcal{X}\times \mathcal{Y}$, $h(q,\cdot)$ is a model with model parameters $q$ such as a neural network or a regression function, and $\ell$ is a loss function.
A common approach within the machine learning community for solving problems of the type given by (\ref{problem:objective_function}) 
is to employ \emph{stochastic gradient descent} (SGD) \citep{robbins_monro_1951}. A solution $q_*$ to (\ref{problem:objective_function}) is approximated iteratively by $q_k$, $k = 1, 2, \ldots$, via
\begin{align}\label{eq:sgd}
q_{k+1} = q_k - \alpha_k \nabla f(q_k,\xi_k),
\end{align}
where $\alpha_k $ is the learning rate, $\nabla f(q_k,\xi_k)$ is a stochastic approximation of $\nabla F(q_k)$, and $\xi_k$ is a random variable that accounts for the stochasticity.
A common choice is to take a random subset $B_{\xi_k} \subset \{1, \dots, N\}$ of the indices of the objective function defined by (\ref{eq:F_sum}) and choose
\begin{align}\label{eq:nablaf_batch}
\nabla f(q,\xi_k) = \frac{1}{|B_{\xi_k}|} \sum_{i \in B_{\xi_k}} \nabla f_i(q),
\end{align}
where $|B_{\xi_k}|$ denotes the cardinality of $B_{\xi_k}$. This is attractive when $N$ is very large and $|B_{\xi}|\ll N$, as it is less computationally expensive than gradient descent. The method also tends to escape local saddle points \citep{fang_et_al_2019} - an appealing property as many machine learning problems are non-convex.
Among the variations of SGD is the popular \emph{SGD with momentum}. Its deterministic counterpart was first introduced in the seminal work of \citet{polyak_1964}. A common form of this algorithm is expressed as an update in two stages
\begin{align}\label{eq:sgd_with_momentum}
\begin{split}
p_{k+1} &= \beta_k p_k - \alpha_k \nabla f(q_k,\xi_k) \\
q_{k+1} &= q_k + \alpha_k p_{k+1}
\end{split}
\end{align}
where $p_0 = 0$ and $\beta_k > 0$ is a momentum parameter.
The usage of the momentum update makes the algorithm less sensitive to noise. Indeed, by an iterative argument, we obtain that
$p_{k+1} = - \sum_{i=0}^k \left( \prod_{j=i+1}^k \beta_j \right) \alpha_i \nabla f(q_i,\xi_i).
$
 That is, $p_{k+1}$ is an average of the previous gradients, where $\beta_k$ determines how much we value information from the preceding stages.

Notwithstanding the benefits of stochastic gradient algorithms, they frequently suffer from instability problems such as exploding gradients \citep{pascanu_2013, bengio_1994} and sensitivity to the choice of learning rate \citep{owens_filkin_1989}.
A way to mitigate these issues is to employ gradient clipping \citep{Goodfellow-et-al-2016,pascanu_2013}
or gradient normalization. 
Gradient normalization was introduced in \citet{poljak_1967} in the deterministic setting and a stochastic version appears already in \citet{andradottir_1990}. A normalized version of the algorithm determined by (\ref{eq:sgd}) is given by
\begin{align*}
q_{k+1} = q_k - \alpha_k \frac{\nabla f(q_k,\xi_k)}{\lVert \nabla f(q_k,\xi_k) \rVert_2}.
\end{align*}
In practice a small number $\epsilon >0$ is added in the denominator to ensure that the update does not become infinitely large.

Gradient clipping was first introduced in \citet{mikolov_2013}. In so-called \emph{hard clipping}, the gradient is simply rescaled if it is larger than some predetermined threshold. \emph{Soft clipping}, on the other hand, makes use of a differentiable function for rescaling the gradient \citep{zhang_2020}. It was recently shown that hard clipping algorithms suffer from an \emph{unavoidable bias term}  
\citep{koloskova_et_al_2023}; a term in the convergence bound that does not decrease as the number of iterations increases. This is one reason why soft clipping is preferable.

\subsection{Gradient normalization, momentum and Hamiltonian systems}
In this article, we study gradient normalization and soft clipping of stochastic momentum algorithms from the perspective of \emph{Hamiltonian systems}. As a first step, we note that if we take $\beta_k = 1 - \gamma \alpha_k$ with $\gamma >0$, we can view the scheme given by~(\ref{eq:sgd_with_momentum}) as an approximate implicit-explicit Euler discretization of the equation system
\begin{align}\label{eq:sgd_mometum_ode}
\begin{split}
\dot{p} &= - \nabla F(q) - \gamma p,\\
\dot{q} &= p.
\end{split}
\end{align}
The system (\ref{eq:sgd_mometum_ode}) is \emph{nearly Hamiltonian} \citep{glendinning_1994}; taking 
\begin{align}\label{eq:separable_Hamiltonian}
H(p,q) = F(q) + \varphi( p),
\end{align}
with $\varphi(p)= \frac{1}{2}\lVert p \rVert_2^2$, we can write it on the form
\begin{align}\label{eq:nearly_Hamiltonian_system}
\begin{split}
\dot{p} &= - \nabla_q H(p,q) -  \nabla_{\dot{q}} \mathcal{R}(\dot{q}),\\
\dot{q} &= \nabla_p H(p,q).
\end{split}
\end{align}
Here, $\nabla_p,\nabla_q$ denote the gradients with respect to $p$ and $q$ respectively and $\mathcal{R}(\dot{q}) = \gamma \frac{\lVert \dot{q}\rVert_2^2}{2}$ is a \emph{Rayleigh dissipation function} that  accounts for energy dissipation (viscous friction) of the system. Note that this choice of $\mathcal{R}$ yields $\nabla_{\dot{q}} \mathcal{R}(\dot{q}) = \gamma \nabla_p H(p,q)$, which will always be the case in this paper. Thus, for a Hamiltonian of the form (\ref{eq:separable_Hamiltonian}), (\ref{eq:nearly_Hamiltonian_system}) reads
\begin{align*}
\begin{split}
\dot{p} &= - \nabla F(q) -  \gamma \nabla \varphi(p),\\
\dot{q} &= \nabla \varphi(p).
\end{split}
\end{align*}
We notice that any \emph{fixed point} of this system is a stationary point of $F$, since $(\dot{q},\dot{p}) = 0$ implies that $\nabla F(q)=0$.
The dissipation term is often included as an extra term in the Euler-Lagrange equations
\begin{align*}
\nabla_{\dot{q}} L(q,\dot{q}) - \frac{\mathrm{d}}{\mathrm{d}t} L(q,\dot{q}) = \nabla_{\dot{q}} \mathcal{R}(\dot{q}),
\end{align*}
where $L(q,\dot{q}) = \varphi^*(\dot{q}) - F(q)$ is the \emph{Lagrangian}, and $\varphi^*$ is the convex conjugate of $\varphi$, compare Proposition~51.2 and Ex.~51.3 in \citet{zeidler3_1985}. The physical interpretation is that 
$q$ is the position of a particle in a potential field $F(q)$ with kinetic energy given by $\varphi^*(\dot{q})$ (in the case when $\varphi(p) = \frac{\lVert p \rVert_2^2}{2}$ we have $\varphi^*(\dot{q}) = \frac{\lVert \dot{q}\rVert_2^2}{2}$).
In many scenarios, such as in this case,  it happens that the friction term is proportional to the velocity \citep{goldstein_mechanics}. A ball rolling on a rough incline \citep{wolf_et_al_1998, bideau_et_al_1994}) or on a tilted plane coated with a viscous fluid \citep{bico_et_al_2009} could for instance be modelled in this fashion, giving weight to the analogy of the \emph{heavy ball} \citep{polyak_1964}. See also \citet{Goodfellow-et-al-2016}, for a further discussion on this.

In this paper, we consider generalizations of the algorithm defined by (\ref{eq:sgd_with_momentum}) to equations of the type (\ref{eq:nearly_Hamiltonian_system}), where
$\varphi$ is a convex, coercive and $L$-smooth function. The family of schemes we consider are given by
\begin{align}\label{eq:stochastic_Hamiltonian_descent}
\begin{split}
p_{k+1} &= p_k - \alpha_k \nabla f(q_k,\xi_k) - \alpha_k \gamma \nabla \varphi(p_k),\\
q_{k+1} &= q_k + \alpha_k \nabla \varphi(p_{k+1}),
\end{split}
\end{align}
where $p_0=0$, $q_0$ is arbitrary, and $\{\xi_k\}_{k\geq 0}$ is a sequence of independent, identically distributed random variables. We show that these schemes converge almost surely to the set of stationary points of $F$.
If we take $\varphi(x) = \frac{\lVert x \rVert_2^2}{2}$ in (\ref{eq:stochastic_Hamiltonian_descent}), we get SGD with momentum~(\ref{eq:sgd_with_momentum}). Taking $\varphi(x) = \sqrt{\lVert x \rVert_2^2 + \epsilon}, \ \epsilon >0$, gives us a gradient normalization scheme, where both the gradient and the momentum variables are rescaled:
\begin{align}\label{eq:normalized_sgd_with_momentum}
\begin{split}
p_{k+1} &= p_k - \alpha_k \nabla f(q_k,\xi_k) - \alpha_k \gamma \frac{ p_k }{\sqrt{\lVert  p_k \rVert_2^2 + \epsilon}},\\
q_{k+1} &= q_k + \alpha_k \frac{ p_{k+1} }{\sqrt{\lVert  p_{k+1} \rVert_2^2 + \epsilon}}.
\end{split}
\end{align}
Other conceivable choices are
\begin{enumerate}[label=\roman*),start=1]
\item \emph{Relativistic kinetic energy:}
 $\varphi(x) = c \sqrt{\lVert x \rVert_2^2 + (mc)^2}$. \citep{franca_et_al_2020} \label{ex:relativistic}
\item \emph{Non-relativistic kinetic energy:} $\varphi(x) = \frac{1}{2}\langle x, A x \rangle + \langle b,x \rangle + c$, where $A$ is a positive definite, symmetric matrix, $b \in \R^d$ and $c\in \R$. \citep{goldstein_mechanics} \label{ex:nonrelativistic}
\item \emph{Gradient rescaling:}
 $\varphi(x) = c \sqrt{\lVert x \rVert_2^2 + \epsilon}$, for $c,\epsilon >0$. 
 \label{ex:rescaling}
 \item \emph{Soft clipping:}
 $\varphi(x) =  \sqrt{1 + \lVert x \rVert_2^2}$.  \label{ex:softclipping}
\item \emph{The symmetric LogSumExp-function:}
 $\varphi(x) = \ln \left( \sum_{i=1}^d e^{x_i} + e^{-x_i} \right)$, which can be seen as an approximation of the $\ell^{\infty}$-norm \citep{sherman_2013}. \label{ex:logsumexp}
\item \emph{Half-squared $\ell^p$-norm:} $\varphi(x) = \frac{1}{2} \lVert x \rVert_p^2$, for $p \in [2,\infty)$.  \label{ex:lpnorm}
\end{enumerate}
Examples \ref{ex:relativistic}, \ref{ex:rescaling} and \ref{ex:softclipping} are analytically similar, but give rise to different behaviours in the algorithm given by (\ref{eq:stochastic_Hamiltonian_descent}).
We refer the reader to \citet{beck_2017,peressini1993}, for verifying that the functions above satisfy the assumptions in Section~\ref{sec:setting}.

\section{Contributions}
Making use of Hamiltonian dynamics, we consider a large class of stochastic optimization algorithms~(\ref{eq:stochastic_Hamiltonian_descent}) for large-scale optimization problems, for which we perform a rigorous convergence analysis. Our assumptions on the dissipation term $\varphi$ are fairly permissive, and thus the class of algorithms covers both interesting cases like normalized SGD with momentum and various soft-clipping methods with momentum, as well as novel methods. Our analysis shows that the iterates generated by any method in this class are finite almost surely, and that they converge almost surely to the set of stationary points of the objective function $F$. This means that the methods ``always'' work in practice, in contrast to what can be guaranteed by analyses that show convergence in expectation. These results are valid in many applications, due to fairly weak assumptions on the optimization problem. The exact assumptions are listed in Section~\ref{sec:analysis} but essentially consist of either
\begin{itemize}
\item  $L$-smooth objective functions and stochastic gradients with possibly infinite variance, or
\item $(L_0,L_1)$-smooth objective functions and heavy-tailed stochastic gradients with bounded variance, or
\item $(L_0,L_1)$-smooth objective functions arising in the empirical risk minimization setting and stochastic gradients with possibly infinite variance but bounded expectation.
\end{itemize}
In particular, we do not assume convexity of the objective function $F$ in any of the cases.

\section{Outline}
In Section~\ref{sec:related}, we briefly discuss some results that are related to the analysis in this paper. The main results and analysis is presented in Section~\ref{sec:analysis}, with conclusions in Section~\ref{sec:conclusions}. The details of the analysis can be found in Appendix~\ref{appendix:analysis}. This depends on some auxiliary results listed in Appendices~\ref{appendix:auxiliary}, \ref{appendix:existence} and \ref{appendix:arzela_ascoli}.

\section{Related works}\label{sec:related}
In the first subsection we consider other formulations of SGD with momentum and how the formulation in this 
paper relates to them. 
In the second subsection we summarize work in optimization and statistics which make use of Hamiltonian dynamics.
Next, we discuss the approach we use for showing almost sure convergence of the methods. Finally, we discuss the central $(L_0,L_1)-$smoothness condition on the objective function.

\subsection{Momentum algorithms}
The implementations of SGD with momentum in the libraries Tensorflow \citep{TensorFlow} and Pytorch \citep{PyTorch} given by
\begin{align*}
\begin{split}
p_{k+1} &= \beta_k p_k - \alpha_k \nabla f(q_k,\xi_k) \\
q_{k+1} &= q_k +  p_{k+1}
\end{split}
\end{align*}
are equivalent to (\ref{eq:sgd_with_momentum}) after the transformation $\alpha_k \to \sqrt{\alpha_k}$, $\beta_k \to \beta_k \sqrt{\frac{\alpha_{k+1}}{\alpha_{k}}}$.
The method (\ref{eq:sgd_with_momentum}) also resembles the (hard-clipped) scheme proposed in \citet{mai_johansson_2021}:
\begin{align*}
\begin{split}
p_{k+1} &= \text{clip}_{r} \left( (1- \beta_k)p_k - \beta_k \nabla f(q_{k},\xi_k) \right),\\
q_{k+1} &= q_k  +\alpha_k p_{k+1},
\end{split}
\end{align*}
where $\text{clip}_{r}$ is a projection operator that projects the argument onto a ball of radius $r$ at the origin.
Further, (\ref{eq:sgd_with_momentum}) is reminiscent of
Stochastic Primal Averaging (SPA) \citep{defazio2021}:
\begin{align*}
p_{k+1} &= p_k - \eta_k \nabla f(q_k,\xi_k),\\
q_{k+1} &= (1- c_{k+1})q_k + c_{k+1} p_{k+1}.
\end{align*}
In Theorem 1~in~\citet{defazio2021} it is shown that this is equivalent to the SGD with momentum version
\begin{align*}
p_{k+1} &= \beta_k p_k + \nabla f(q_k,\xi_k),\\
q_{k+1} &= q_k - \alpha_k p_{k+1},
\end{align*}
if one takes  $\eta_{k+1} = \frac{\eta_k - \alpha_k}{\beta_{k+1}}$ 
and $c_{k+1} = \frac{\alpha_k}{\eta_k}$.
The SPA algorithm can be seen as a randomized implicit-explicit Euler discretization of the equation system
\begin{align*}
\dot{p} &=  -\nabla F(q),\\
\dot{q} &=  p-q,
\end{align*}
which after a change of variable is equivalent to (\ref{eq:sgd_mometum_ode}) for $\gamma=1$.
Under the rather strong assumption that the noise is almost surely bounded (which does not hold for, e.g., Gaussian noise), so-called mixed-clipped SGD with momentum was studied in \citet{zhang_2020}:
\begin{align*}
p_{k+1} &= \beta p_k - (1- \beta) \nabla f(q_k,\xi_k),\\
q_{k+1} &= q_k - 
\nu \min \left( \eta, \frac{\gamma}{\lVert p_{k+1} \rVert_2} \right)p_{k+1} \\
 &\qquad\;+
 (1-\nu) \min \left( \eta, \frac{\gamma}{\lVert \nabla f(q_k,\xi_k)  \rVert_2} 
 \right)
 \nabla f(q_k,\xi_k)
,
\end{align*}
Here, $0 \leq \nu \leq 1$ is an interpolation parameter.

A drawback with the previously mentioned analyses is that the convergence results are obtained in expectation, which means that there is no guarantee that a single path will converge.

\subsection{Hamiltonian dynamics}
Hamiltonian dynamics, in its energy conserving form, has been well-explored in the Markov chain Monte Carlo field, compare \cite{leimkuhler_2015}. In \citet{livingstone_2017}, various kinetic energy functions $\varphi$ are considered for equation (\ref{eq:nearly_Hamiltonian_system}) without the dissipation term $\nabla_{\dot{q}} \mathcal{R}(\dot{q})$.

The algorithm (\ref{eq:stochastic_Hamiltonian_descent}) was studied in the context of stochastic differential equations and Langevin dynamics in
\citet{stoltz_trstanova_2018}, where the noise is assumed to be Gaussian. In general, this is however a restrictive assumption in the stochastic optimization setting. 

The specific update 
(\ref{eq:normalized_sgd_with_momentum}) bears resemblance to deterministic time integration-  and optimization schemes studied in \citet{franca_et_al_2020}, that arise as discretizations of the system
\begin{align*}
\dot{p} &= - \nabla_q H(p,q) - \gamma p,\\
\dot{q} &= \nabla_p H(p,q),
\end{align*}
where the dissipation term $\gamma p$ emanates from \emph{Bateman's Lagrangian} $L(q,\dot{q}) = e^{\gamma t}(\varphi^*(\dot{q}) - F(q))$, see \citet{bateman_1931}. A similar point of view is also taken in \cite{Franca2020O}, but where so-called Bregman dynamics is employed.
In the (deterministic) optimization setting this was studied in \citet{maddison_et_al_2018}, where strictly convex kinetic energy functions $\varphi$ are considered.
A stochastic gradient version is analysed in \citet{Harshvardhan} for strongly convex objective functions $F$.

However, the stochastic optimization algorithm has not been studied for non-convex problems, and an analysis for merely convex (and not strictly convex) kinetic energy functions is lacking. 

\subsection{Almost sure convergence}

The analysis in this paper is based on the \emph{ODE method}, emanating from \citet{ljung_1976}.
The particular proof strategy is due to 
\cite{kushner_clark_1978}, and is based on linear interpolation of the sequence of iterates.
The technique was extended to piecewise constant interpolations in \cite{kushner_yin_2003}.
The approach relies on the assumption that the iterates generated by the algorithm are finite almost surely; an assumption that has to be verified independently.

A similar analysis of the SGD with momentum was performed in \cite{gadat_et_al_2018}. It was extended in \citet{barakat_et_al_2021}, to a class of schemes that encompasses (\ref{eq:sgd_with_momentum}).
The analytical approach is slightly different and does not cover the normalization- and clipping algorithms that we analyze in this article.

We also note that one can employ an analysis similar to that in e.g.\ \cite{BottouCurtisNocedal.2018}, along with martingale results like that in \cite{robbins_siegmund} to obtain almost sure convergence of a subsequence of the iterates. This is for instance the case in \cite{sebbouh_et_al} where almost sure convergence guarantees of the type $\min_{0 \leq k \leq K } \lVert \nabla F(q_k) \rVert_2 \to 0$ almost surely for SGD and SGD with momentum are established. These types of results are weaker than those obtained in this paper, since they cannot guarantee that the whole sequence of iterates $\{q_k\}_{k \geq 0}$ converges to a stationary point.

\subsection{$(L_0,L_1)-$smoothness}
The $(L_0,L_1)-$smoothness assumption was introduced in \cite{Zhang2020Why} as a more appropriate measure of smoothness than $L$-smoothness for certain machine learning problems. It is shown in \cite{Zhang2020Why} that
the iteration complexity of clipped SGD is bounded, under the assumption that the stochastic gradients are bounded almost surely. The latter is a very restrictive assumption that is not fulfilled even by Gaussian noise. In \cite{zhang_2020} a clipped algorithm with momentum is shown to converge in expectation to a stationary point under the same strong assumptions on the noise. Similar assumptions are also encountered in e.g. \cite{crawshaw2022robustness, li2023convergence}; the latter also considers the slightly more general case of \emph{sub-Gaussian} noise.
\cite{koloskova_et_al_2023} analyses clipped SGD under Assumption \ref{ass:setting_two}.\ref{ass:variance}, but do not obtain a convergence guarantee due to an \emph{unavoidable bias} \citep{koloskova_et_al_2023}.
Recently \cite{WangZM023} and \cite{FawRCS23} obtained convergence guarantees for versions of AdaGradNorm under the weaker affine variance-assumption. These results are however only with a certain probability, and there is always some set of positive measure on which the algorithm may not converge.

The convergence guarantees that we obtain in Theorem 
\ref{thm:main}
under Assumption \ref{ass:setting_two}.\ref{ass:L0L1smoothness} and Assumption \ref{ass:setting_one}.\ref{ass:variance} is stronger in the sense that it converges for every path.
We also stress the fact that Assumption \ref{ass:setting_one}.\ref{ass:variance}
is relatively weak since it covers all heavy-tailed distributions with finite variance \citep{rolski2009stochastic}. This includes for instance the large class of sub-Weibull distributions, which generalizes sub-Gaussian and sub-exponential distributions \citep{vladimirova2020sub}.

\section{Analysis}\label{sec:analysis}
We first give a brief overview of the analysis in Section \ref{sec:overview}. In Section \ref{sec:setting} we describe the setting and
in Section \ref{sec:outline_of_proof} we give a more detailed outline of the theorems and the proofs.
The full proofs of the results are given in Appendix \ref{appendix:analysis}.
\subsection{Brief overview}\label{sec:overview}

The analysis is split into two parts. 

In the first, we show that the iterates of the scheme defined by (\ref{eq:stochastic_Hamiltonian_descent}) are finite almost surely, if the objective function $F$ and the convex kinetic energy function $\varphi$ are $L$-smooth and coercive, or if $F$ is $(L_0,L_1)-$smooth and the variance is finite. This is done by constructing a Lyapunov function with the help of the Hamiltonian $H$, and then appealing to the classical Robbins--Siegmund theorem \citep{robbins_siegmund}.

In the second, we show that given that
the iterates defined by (\ref{eq:stochastic_Hamiltonian_descent}) are bounded, they converge almost surely to a stationary point of $F$. We make use of a modification of the \emph{ODE method}, compare \cite{kushner_yin_2003}. Since the scheme is implicit-explicit, we cannot directly apply e.g.\ Theorem 2.1 in \citet{kushner_yin_2003}.

Essentially, the idea is to 
\begin{enumerate}[label=\roman*),start=1]
\item Introduce a pseudo-time $t_k = \sum_{i=0}^{k-1}\alpha_i$ and construct piecewise constant interpolations $P_0(t)$ and $Q_0(t)$ of $\{p_k\}_{k \geq 0}$ and $\{q_k\}_{k \geq 0}$ from (\ref{eq:stochastic_Hamiltonian_descent}).
\item Show that the time shifted processes $P_k(t)=P_0(t_k+t)$ and $Q_k(t)=Q_0(t_k+t)$ are \emph{equicontinuous in the extended sense} \citep{kushner_yin_2003} and that $P_k(t)$ and $Q_k(t)$ asymptotically satisfies (\ref{eq:nearly_Hamiltonian_system}).
\item At last, make use of the underlying dynamics of (\ref{eq:nearly_Hamiltonian_system}) to conclude that $\{q_k\}_{k \geq 0}$ converges almost surely to a stationary point of $F$.
\end{enumerate}

\subsection{Setting}\label{sec:setting}
Let $(\Omega, \mathcal{F}, \Prob)$ be a probability space, and $\{\xi_k\}_{k\geq 0}$ be a sequence of independent, identically distributed random variables. We further let $\mathcal{F}_k $ denote the $\sigma-$algebra generated by $\xi_0,\dots, \xi_{k-1}$. 
By $\E_{\xi_k} \left[ X \right]$ we denote the conditional expectation of a random variable $X$ with respect to $\mathcal{F}_k$.
For a set $A\subset \mathbb{R}^d$, we let 
$N_{\delta}(A) = \{x: \inf_{a \in A}\| x -a \| < \delta\}$.
\subsubsection{Basic assumptions}

We make the following basic assumptions on $f$, $F$ and $\varphi$:
\begin{assumption}\label{ass:basic}
The objective function $F$ is differentiable and satisfies:
\begin{enumerate}[label=\roman*),start=1]
\item\label{ass:coercive} \emph{(Coercivity)} $\lim_{\lVert x \rVert_2 \to \infty} F(x) = \infty$. 
\item\label{ass:proper} \emph{(Proper)} There is a number $F_*> - \infty$ such that $F(x) \geq F_*, \ \forall x \in \R^d$. 
\item\label{ass:critical_points_locallyfinite}  \emph{(Locally finite cardinality)} Let $\Lambda = \{q: \nabla F(q)=0\}$. For every compact set $K \subset \mathbb{R}$, the set $F(\Lambda) \cap K$ has finite cardinality. 
\end{enumerate}

Further, the stochastic gradient $\nabla f$ is an unbiased estimator of $\nabla F$, i.e.
\begin{enumerate}[label=\roman*),start=4]
\item\label{ass:unbiased} $\mathbb{E} \left[ \nabla f(x,\xi) \right]=\nabla F(x)$. 
\end{enumerate}

\end{assumption}

\begin{remark}\label{remark:coercive}
  Assumption \ref{ass:basic}.\ref{ass:coercive} implies that the sublevel sets $\{x: F(x)\leq c\}$ are bounded, compare Proposition~11.12 in \citet{bauschke_2011}. 
\end{remark}

\begin{remark}\label{remark:critical_points_locallyfinite}
Assumption \ref{ass:basic}.\ref{ass:critical_points_locallyfinite} is slightly more general than the assumption that $F(\Lambda)$ has finite cardinality, which one often sees; compare e.g. \cite{benaim_1996}. We make use of it in Lemma \ref{lemma:locally_asymptotically_stable}, in order to show that the sublevel sets of the Hamiltonian are locally asymptotically stable.
\end{remark}

\begin{assumption}\label{ass:kinetic_energy_function}
The kinetic energy function $\varphi$ is differentiable and satisfies:
\begin{enumerate}[label=\roman*),start=1]
\item\label{ass:Lipschitzcontinuous_gradient_phi} \emph{(Lipschitz continuous $\nabla \varphi$)} There is a constant $\lambda>0$ such that  $\lVert \nabla  \varphi(y) - \nabla \varphi(x)\rVert_2 \leq \lambda \lVert x -y \rVert_2$, for all $ x,y \in \R^d$. 

\item\label{ass:convexity} \emph{(Convexity)} It holds that 
$\varphi(y) -\varphi(x) \leq \langle \nabla \varphi(y),y-x \rangle$ for all $x,y \in \R^d$. 
\item\label{ass:coercive_kinetic} \emph{(Coercivity)} $\lim_{\lVert x \rVert_2 \to \infty} \varphi(x) = \infty$. 
\item\label{ass:proper_kinetic} \emph{(Proper)} For all $x \in \R^d$, it holds that $\varphi(x) \geq \varphi_*>-\infty$. 
\end{enumerate}
\end{assumption}

\begin{remark}\label{remark:phi_coercive}
Asumption \ref{ass:basic}.\ref{ass:coercive}, \ref{ass:basic}.\ref{ass:proper}, \ref{ass:kinetic_energy_function}.\ref{ass:coercive_kinetic} and \ref{ass:kinetic_energy_function}.\ref{ass:proper_kinetic} together implies that the Hamiltonian $H(p,q) = F(q) + \varphi(p)$ is coercive as a function of $q$ and $p$.
\end{remark}

In addition to these basic assumptions, we consider three different settings. 

\subsubsection{Setting 1}
In the first setting, $\nabla F$ is Lipschitz continuous but the stochastic gradients can have large variance.
\begin{assumption}\label{ass:setting_one}
The objective function $F$ and the stochastic gradient $\nabla f$ further satisfy:

\begin{enumerate}[label=\roman*),start=1]
\item\label{ass:Lipschitzcontinuous_gradient} \emph{(Lipschitz-continuous $\nabla F$)} There is a constant $L>0$ such that  $\lVert \nabla F(y) - \nabla F(x)\rVert_2 \leq L \lVert x -y \rVert_2$, for all $ x,y \in \R^d$. 
\item\label{ass:expectedsmoothness} \emph{(Locally bounded variance)}
  \begin{equation*}
  \mathbb{V} \left[  \nabla f(x,\xi)  \right] \leq \kappa \left(F(x) - F_* \right) + \tau \lVert \nabla F(x) \rVert_2^2 +\sigma^2,
\end{equation*}
\end{enumerate}
where $\kappa, \sigma, \tau \geq 0$. 

\end{assumption}
Assumption \ref{ass:setting_one}.\ref{ass:Lipschitzcontinuous_gradient} implies that the inequality $
F(y) - F(x) \leq \langle \nabla F(x), y-x \rangle + \frac{L}{2}\lVert x -y \rVert_2^2$ holds
for all $x,y \in \R^d$, compare Lemma~1.2.3~in~ \cite{nesterov_2018}.
Assumption \ref{ass:setting_one}.\ref{ass:expectedsmoothness} was first introduced in \cite{khaled2023bettertheorysgdnonconvex} where it was called ``expected smoothness''. It is weak in the sense that it allows for infinite variance in the case that either the gradient or the objective function becomes infinitely large. 
The condition is similar to e.g. the ``affine noise variance'' in \cite{pmlr-v195-wang23a} and the ``affine variance'' in \cite{FawRCS23}.

\subsubsection{Setting 2}
Alternatively, we consider the following setting, where we require less regularity of $F$ but instead restrict the variance of the stochastic gradient.

\begin{assumption}\label{ass:setting_two}
The objective function $F$ and the stochastic gradient $\nabla f$ further satisfy:

\begin{enumerate}[label=\roman*),start=1]
\item\label{ass:L0L1smoothness} \emph{($(L_0,L_1)-$smoothness)}
There exists $L_0, L _1$ such that for all $x,y \in \mathbb{R}^d$, if $\lVert x -y \rVert_2 \leq \frac{1}{L_1}$, then
\begin{equation*}
\lVert \nabla F(x) - \nabla F(y) \rVert 
\leq \left( L_0 + L_1 \lVert \nabla F(y)\rVert_2\right)\lVert x- y \rVert_2.
\end{equation*}

\item\label{ass:variance} \emph{(Bounded variance)} $\mathbb{V} \left[  \nabla f(x,\xi)  \right] \leq \sigma^2$,
\item\label{ass:bounded_clipping} \emph{(Bounded $\nabla \varphi$)}
There exists $\Delta > 0$ such that $\lVert \nabla \varphi(x) \rVert_2 
\leq \Delta$ for all $x \in \R$. 
\end{enumerate}
  Here $\mathbb{V} \left[  \nabla f(x,\xi)  \right] = \E \left[  \lVert \nabla f(x,\xi) \rVert_2^2 \right] - \lVert \nabla F(x) \rVert_2^2$. 
\end{assumption}

\begin{remark}\label{remark:lyapunov}
By \emph{Lyapunov's inequality}, compare p.~230~in~\cite{shiryaev_2016}, it follows from Assumption \ref{ass:setting_two}.\ref{ass:variance} that
$\mathbb{E} \left[ \lVert  \nabla f(x,\xi)-\nabla F(x) \rVert_2  \right] \leq \sigma.$
\end{remark}
Assumption \ref{ass:setting_two}.\ref{ass:variance} is sometimes referred to as a \emph{heavy-tailed noise} assumption in the literature \citep{gorbunov2020stochastic, koloskova_et_al_2023}. It covers all zero-mean, heavy-tailed distributions with finite second moment, compare \cite{rolski2009stochastic}. In particular, it also includes the large class of \emph{sub-Weibull distributions}
\citep{vladimirova2020sub}, which generalizes random variables of sub-Gaussian and sub-Exponential distribution.

\subsubsection{Setting 3}
Assumption \ref{ass:setting_two}.\ref{ass:variance} may be restrictive in some cases, compare \citep{gurbuzbalaban21_2021}. Therefore we also consider the the empirical risk minimization setting when the objective function is on the form (\ref{eq:F_sum}), $f(\cdot,\xi)$ is $(L_0,L_1)-$smooth and $\nabla f(\cdot,\xi)$ is given by (\ref{eq:nablaf_batch}). In this setting, we can 
further lower the assumptions on the noise to merely finite expectation:
\begin{assumption}\label{ass:setting_three}
The objective function satisfies \ref{ass:setting_two}.\ref{ass:L0L1smoothness} and $\nabla \varphi$ satisfies \ref{ass:setting_two}.\ref{ass:bounded_clipping}.
The objective function $F$ and the stochastic gradient $\nabla f$ further satisfy:
\begin{enumerate}[label=\roman*),start=1]
\item\label{ass:empiricalriskminimization} \emph{(Empirical risk minimization)} The objective function and the stochastic gradient are of the form (\ref{eq:F_sum}) and (\ref{eq:nablaf_batch}),
where for each $i$ it holds that $\inf_{q \in \mathbb{R}^d} f_i(q) > - \infty$.
\item\label{ass:heavytailednoise} \emph{(Bounded expectation)} The stochastic gradients satisfy
  \begin{equation*}
\mathbb{E}\left[ \lVert \nabla f(x,\xi) - \nabla F(x) \rVert_2 \right] \leq \sigma.
\end{equation*}
\item\label{ass:stochasticfunctionsL0L1smooth} \emph{($(L_0,L_1)-$smooth)}  
For each $\xi$, $f(\cdot,\xi)$ is $(L_0,L_1)$-smooth. 
\end{enumerate}
\end{assumption}

\begin{remark}
  We observe that SGD with momentum, corresponding to~(\ref{eq:sgd_with_momentum}), requires $\varphi(x) = \lVert x \rVert^2 / 2$, which does not satisfy Assumption~\ref{ass:setting_two}.\ref{ass:bounded_clipping}. It is thus only covered by the analysis in Setting 1. In Setting 2 and 3, we have a weaker regularity assumption on $F$, and this requires us to instead pose stricter requirements on the methods. Overall this indicates that for $(L_0,L_1)-$smooth cost functionals, clipping methods are a better option than SGD.
\end{remark}

\subsubsection{Book-keeping assumptions}
The following assumption on the step sizes $\alpha_k$ is standard and originates from \citet{robbins_monro_1951}. Informally, the step sizes must go to zero in order to counter the stochasticity, but do so slowly enough that we have time to reach a stationary point.
\begin{assumption}[Step sizes]\label{ass:step_size}
The step size sequence $\{\alpha_k\}_{k \geq 0}$ satisfies $\alpha_0 = 0$ and 
$\{\alpha_k\}_{k \geq 0} \in \ell^2(\mathbb{R}) \backslash \ell^1(\mathbb{R})$.
\end{assumption}

Our analysis shows convergence to the set of stationary points of $\nabla F$. Under the following additional assumption, we get convergence to a unique stationary point:
\begin{assumption}\label{ass:isolated_equilibria}
The stationary points of $F$ are isolated.
\end{assumption}

\subsection{Outline of proof}\label{sec:outline_of_proof}
The proofs of the results in this section can be found in Appendix \ref{appendix:analysis}.
The main theorem is an extension of the approach in \cite{kushner_yin_2003}:
\begin{restatable}{theorem}{thmmain}\label{thm:main}
Let Assumptions~\ref{ass:basic}, \ref{ass:kinetic_energy_function} and~\ref{ass:step_size} be satisfied, as well as either Assumption~\ref{ass:setting_one}, \ref{ass:setting_two} or \ref{ass:setting_three}. Then $\{q_k\}_{k \geq 0}$ converges almost surely to the set of stationary points of the objective function $F$. 
If we additionally assume that Assumption~\ref{ass:isolated_equilibria} holds, the convergence is to a unique stationary point.
\end{restatable}
The following result is a direct consequence of Theorem~\ref{thm:main}:
\begin{restatable}[Convergence in expectation]{corollary}{corconvinexp}
\label{cor:convinexp}
Let Assumptions~\ref{ass:basic}, \ref{ass:kinetic_energy_function},  \ref{ass:step_size} and \ref{ass:isolated_equilibria} be valid. Further, let the Hamiltonian be on the form (\ref{eq:separable_Hamiltonian}) and let the sequences $\{p_k\}_{k\geq 0}$ and $\{q_k\}_{k\geq 0}$ be generated by~(\ref{eq:stochastic_Hamiltonian_descent}).
Then it holds that
\begin{align*}
\lim_{k \to \infty} \mathbb{E} \left[ 
\lVert \nabla F(q_k) \rVert_2^{\theta} \right] = 0,
\end{align*}
where $\theta = 1$ under Assumption~\ref{ass:setting_one} and 
$\theta = \frac{1}{2}$ under Assumption~\ref{ass:setting_two} or \ref{ass:setting_three}.
\end{restatable}
Our proof strategy consists of two parts. In the first part we show that the sequences $\{p_k\}_{k \geq 0}$ and $\{q_k\}_{k \geq 0}$ are finite almost surely:
\begin{restatable}[Finiteness of $\{p_k\}_{k\geq 0}$ and $\{q_k\}_{k\geq 0}$]{theorem}{thmpkqkbdd}
\label{thm:pqbounded}
Let Assumptions~\ref{ass:basic}, \ref{ass:kinetic_energy_function} and \ref{ass:step_size} be valid, as well as either Assumption~\ref{ass:setting_one},~\ref{ass:setting_two} or \ref{ass:setting_three}. Further, let the Hamiltonian be on the form (\ref{eq:separable_Hamiltonian}) and let the sequences$\{p_k\}_{k\geq 0}$ and $\{q_k\}_{k\geq 0}$ be generated by~ (\ref{eq:stochastic_Hamiltonian_descent}).
Then $\{p_k\}_{k\geq 0}$ and $\{q_k\}_{k\geq 0}$ are finite almost surely. Moreover, it holds that $\sup_{k\geq 0} \mathbb{E} \left[F(q_k) \right] < \infty$.
\end{restatable}
In the second part of the analysis, we closely follow the ODE method approach as outlined in~\citet{kushner_yin_2003}: We start by introducing a pseudo-time $t_k = \sum_{i=0}^{k-1} \alpha_i$, and define two piecewise constant, (stochastic) interpolation processes by
\begin{align}\label{def:P0Q0}
\begin{split}
P_0(t) &= p_0 I_{(-\infty,t_{0}]}(t) + \sum_{k=0}^{\infty} p_k I_{[t_k,t_{k+1})}(t), \\
Q_0(t) &= q_0 I_{(-\infty,t_{0}]}(t) + \sum_{k=0}^{\infty} p_k I_{[t_k,t_{k+1})}(t).
\end{split}
\end{align}
We next consider the shifted sequence of processes 
$\{P_k\}_{k\geq 0}$ and $\{Q_k\}_{k\geq 0}$, defined by
\begin{align}
  \label{def:PkQk}  
  \begin{split}
    P_k(t)&=P_0(t_k + t),\\ 
    Q_k(t)&=Q_0(t_k + t).
  \end{split}
\end{align}
We note that $\{P_k\}_{k\geq 0}$ and $\{Q_k\}_{k\geq 0}$ are stochastic processes; they depend on $\omega \in \Omega$\footnote{Here $\omega $ is an \emph{outcome} and $\Omega$ is the \emph{sample space} of the underlying probability space $\left(\Omega, \mathcal{F}, \Prob\right)$.}
through the stochasticity of the sequences $\{p_k\}_{k \geq 0}$ and $\{q_k\}_{k\geq 0}$.
For brevity we will refrain from writing out the dependence on $\omega$.

The next step is to introduce the concept of \emph{extended equicontinuity} \citep{kushner_yin_2003,freise}:
\begin{restatable}[Extended equicontinuity]{definition}{defextendedequicontinuity}\label{def:equicontinuity_in_the_extended_sense}
A sequence of $\R^d$-valued functions $\{f_k\}_{k \geq 0}$, defined on $(-\infty,\infty)$, is said to be \emph{equicontinuous in the extended sense} if $\{ |f_k(0)| \}_{k \geq 0}$ is bounded and for every $T > 0$ and $\epsilon > 0$ there is a $\delta > 0$ such that 
\begin{align}\label{limsup}
\limsup_{k \to \infty} \sup_{0<|t-s| \leq \delta, \ t, s \in [0,T]}|f_k(t) - f_k(s)| \leq \epsilon.
\end{align}
\end{restatable}
Following \citet{freise}, we show that the process $\{Z_k\}_{k\geq 0}= \{(P_k,Q_k)\}_{k\geq 0}$ is equicontinuous in the extended sense:
\begin{restatable}[Equicontinuous in the extended sense]{lemma}{lemmaPkQkequicontinuous}\label{lemma:PkQk_equicontinuous}
Consider $\{Z_k\}_{k\geq 0}=\{(P_k,Q_k)\}_{k\geq 0}$  where the sequences $\{P_k\}_{k\geq 0}$ and $\{Q_k\}_{k\geq 0}$ 
are defined by~(\ref{def:PkQk}) (equivalently, by~(\ref{identity:PkQk})). Suppose that $\{p_k\}_{k\geq 0}$ and $\{q_k \}_{k\geq 0}$ are defined by (\ref{eq:stochastic_Hamiltonian_descent}), and that the Hamiltonian is on the form (\ref{eq:separable_Hamiltonian}). Further, let Assumptions~\ref{ass:basic}, \ref{ass:kinetic_energy_function} and \ref{ass:step_size} be valid, as well as either Assumption~\ref{ass:setting_one},~\ref{ass:setting_two} or \ref{ass:setting_three}.
Then $\{Z_k\}_{k\geq 0}$ is equicontinuous in the extended sense, almost surely.
\end{restatable}
We can then appeal to the \emph{extended/discontinuous Arzel{\`a}--Ascoli theorem} to conclude that $\{Z_k\}_{k\geq 0}$ has a subsequence that converges to a continuous function $z$:
\begin{restatable}[Discontinuous Arzel{\`a}--Ascoli theorem]{theorem}{thmdiscontinuousarzelaascoli}\label{thm:discontinuos_arzela_acoli}
Let $\{f_k\}_{k \geq 0}$ be a sequence of functions, defined on $\R^d$, that is equicontinuous in the extended sense. Then there is a subsequence  $\{f_{n_k}\}_{n_k \geq 0}$ of $\{f_k\}_{k \geq 0}$, that converges uniformly on compact sets to a continuous function.
\end{restatable}
The proof, which can be found in Appendix \ref{appendix:arzela_ascoli}, is based on Theorem~6.2 in
\citet{droniou_eymard_2016} and is a modification of Theorem 16.4 in \citet{billingsley1968}. See also \citep{kushner_yin_2003,freise}.

With this established, we proceed to show that $\{Z_k\}_{k\geq 0}$ is an \emph{asymptotic solution}\footnote{By \emph{Grönwall's inequality} (compare e.g.\ \citet{ethier_kurtz_1986}) this is equivalent to (\ref{def:P0Q0}) being an \emph{asymptotic pseudotrajectory} \citep{benaim_1999} to (\ref{eq:nearly_Hamiltonian_system}).} to (\ref{eq:nearly_Hamiltonian_system}); i.e.\ asymptotically $\{P_k\}_{k\geq 0}$ and $\{Q_k\}_{k\geq 0}$ satisfy (\ref{eq:nearly_Hamiltonian_system}).
More precisely we show
\begin{restatable}[Asymptotic solutions]{lemma}{lemmaPkQkasymptoticsolutions}\label{lemma:PkQk_asymptotic_solutions}
  Under the same assumptions and notation as in Lemma~\ref{lemma:PkQk_equicontinuous},
we can write 
\begin{align}\label{eq:PkQk_asymptotic_solutions}
\begin{split}
P_k(t) = P_k(0) &- \int_0^t\nabla F(Q_k(s)) \mathrm{d}s  -  \gamma \int_0^t\nabla \varphi(P_k(s)) \mathrm{d}s  + M_k(t) + \mu_k(t),\\
 Q_k(t) = Q_k(0)  &+ \int_0^t \nabla \varphi(P_k(s)) \mathrm{d}s  + \nu_k(t) + \kappa_k(t),
 \end{split}
\end{align}
where the functions $\{ M_k\}_{k\geq0}$, $\{\mu_k\}_{k \geq 0}$, $\{\nu_k\}_{k \geq 0}$ and
$\{\kappa_k\}_{k \geq 0}$ converge to $0$ uniformly on compact sets almost surely.
\end{restatable}
It follows that any limit point of $\{Z_k\}_{k\geq 0}$ satisfies
\begin{align}\label{eq:nearly_Hamiltonian_system_integral}
\begin{split}
P(t) &= P(0) - \int_0^t\nabla F(Q(s)) \mathrm{d}s  -  \gamma \int_0^t\nabla \varphi(P(s)) \mathrm{d}s  \\
 Q(t) &= Q(0)  + \int_0^t \nabla \varphi(P(s)) \mathrm{d}s .
 \end{split}
\end{align}
The limits we can extract by appealing to Theorem \ref{thm:discontinuos_arzela_acoli} are continuous. Thus it follows from (\ref{eq:nearly_Hamiltonian_system_integral}) and the fundamental theorem of calculus that they are differentiable and satisfy (\ref{eq:nearly_Hamiltonian_system}).
\begin{remark}
The functions $\{\mu_k\}_{k \geq 0}$ and $\{ \nu_k \}_{k \geq 0}$ are essentially what is left when we have rewritten the sums in (\ref{def:P0Q0}) as integrals.
The functions $\{ M_k\}_{k\geq0}$ account for the difference between $\nabla F(q_k)$ and $\nabla f(q_k,\xi_k)$ and $\kappa_k(t)$ for the implicit discretization in the second equation of (\ref{eq:stochastic_Hamiltonian_descent}).
\end{remark}
\begin{remark}
The convergence ``uniformly on compact sets almost surely'' is to be understood as uniformly on compact sets in $t$ and almost surely in $\omega$. 
For example, for the sequence $\{ M_k\}_{k\geq0}$ we have that for any compact set $K \subset \R$, $\lim_{k \to \infty} \sup_{t \in K} \lVert M_k(t)(\omega)\rVert_2=0$
for almost all $\omega \in \Omega$.
\end{remark}
We recall the definition of a \emph{locally asymptotically stable set} \citep{borkar_2008,kushner_yin_2003}:
\begin{restatable}[Locally asymptotically stable set]{definition}{deflocallyasymptoticallystableset}\label{def:locally_asymptotically_stable}
 A  set $A$ is said to be \emph{Lyapunov
stable} if for any $\epsilon>0$, there exists a $\delta >0$ such that every trajectory initiated
in the $N_{\delta}(A)$ remains in $N_{\epsilon}(A)$. It is \emph{locally asymptotically stable} if every such path ultimately goes to $A$.
\end{restatable}
With this in mind, we show the following theorem, which is essentially an adaptation of Theorem 5.2.1 in \cite{kushner_yin_2003}.

\begin{restatable}{theorem}{thmkushneryin}\label{thm:kushner_yin}
  Under the same assumptions and notation as in Theorem \ref{thm:main}, let $A$ be a locally asymptotically stable set for (\ref{eq:nearly_Hamiltonian_system}). If there exists a compact set in the domain of attraction of $A$ that $\{z_k\}_{k \geq 0}$ visits infinitely often, then $z_k \to A$ almost surely:
  \begin{align*}
    \lim_{k \to \infty} \inf_{a \in A} \lVert z_k -a \rVert_{\ell^2\left(\R^{2d}\right)} =0, \ \text{ a.s.}
  \end{align*}
\end{restatable}

The next step is to prove the following, which gives us specific locally asymptotically stable sets:
\begin{restatable}{lemma}{lemmalocallyasymptoticallystable}\label{lemma:locally_asymptotically_stable}
  Consider the same assumptions and notation as in Theorem~\ref{thm:main}.
For each $c$, if the set $\{z:  H(z)\leq c\}$ is non-empty, it is a locally asymptotically stable set for the solutions to~(\ref{eq:nearly_Hamiltonian_system}).
\end{restatable}
In particular, the set $A = \{z:  H(z) \leq \liminf_k H(z_k)\}$ is locally asymptotically stable. By the properties of $\liminf$, we can also find a compact set which $z_k$ enters infinitely often, and we can therefore apply Theorem~\ref{thm:kushner_yin} to conclude that $z_k \to A$. The final step is to show that this convergence in fact implies convergence to the set of stationary points of $H$, and therefore that $q_k$ converges to the set of stationary points of $F$. Under Assumption \ref{ass:isolated_equilibria} we can additionally conclude that the convergence is to a unique equilibrium.

\section{Conclusions}\label{sec:conclusions}
In this paper, we have shown that the stochastic Hamiltonian descent algorithm (\ref{eq:stochastic_Hamiltonian_descent}), arising as a stochastic explicit-implicit Euler discretization of (\ref{eq:nearly_Hamiltonian_system}), under weak assumptions converges almost surely to the set of stationary points of the objective function $F$.
In the terminology of~\cite{robbins_monro_1951}, this means that the estimator determined by $\{q_k\}$ is a strongly consistent estimator of a stationary point of $F$. (Here, the designated asymptoticity is with respect to the number of iterations instead of the sample size.) Similarly, the result in Corollary~\ref{cor:convinexp} is akin to $\{q_k\}_{k \geq 0}$ being an asymptotically unbiased estimator of a stationary point $q_*$.

\bibliography{refs}

\begin{thebibliography}{69}
\providecommand{\natexlab}[1]{#1}
\providecommand{\url}[1]{\texttt{#1}}
\expandafter\ifx\csname urlstyle\endcsname\relax
  \providecommand{\doi}[1]{doi: #1}\else
  \providecommand{\doi}{doi: \begingroup \urlstyle{rm}\Url}\fi

\bibitem[Abadi et~al.(2015)Abadi, Agarwal, Barham, Brevdo, Chen, Citro,
  Corrado, Davis, Dean, Devin, Ghemawat, Goodfellow, Harp, Irving, Isard, Jia,
  Jozefowicz, Kaiser, Kudlur, Levenberg, Man\'{e}, Monga, Moore, Murray, Olah,
  Schuster, Shlens, Steiner, Sutskever, Talwar, Tucker, Vanhoucke, Vasudevan,
  Vi\'{e}gas, Vinyals, Warden, Wattenberg, Wicke, Yu, and Zheng]{TensorFlow}
M.~Abadi, A.~Agarwal, P.~Barham, E.~Brevdo, Z.~Chen, C.~Citro, G.~Corrado,
  A.~Davis, J.~Dean, M.~Devin, S.~Ghemawat, I.~Goodfellow, A.~Harp, G.~Irving,
  M.~Isard, Y.~Jia, R.~Jozefowicz, L.~Kaiser, M.~Kudlur, J.~Levenberg,
  D.~Man\'{e}, R.~Monga, S.~Moore, D.~Murray, C.~Olah, M.~Schuster, J.~Shlens,
  B.~Steiner, I.~Sutskever, K.~Talwar, P.~Tucker, V.~Vanhoucke, V.~Vasudevan,
  F.~Vi\'{e}gas, O.~Vinyals, P.~Warden, M.~Wattenberg, M.~Wicke, Y.~Yu, and
  X.~Zheng.
\newblock {TensorFlow}: Large-scale machine learning on heterogeneous systems,
  2015.
\newblock URL \url{https://www.tensorflow.org/}.

\bibitem[Andradóttir(1990)]{andradottir_1990}
S.~Andradóttir.
\newblock A new algorithm for stochastic optimization.
\newblock In O.~Balci, R.~Sadowski, and R.~Nance, editors, \emph{Proceedings of
  the 1990 Winter Simulation Conference}, pages 364--366, 1990.

\bibitem[Asic and Adamovic(1970)]{asic_adamovic_1970}
M.~D. Asic and D.~D. Adamovic.
\newblock Limit points of sequences in metric spaces.
\newblock \emph{Am. Math. Mon.}, 77\penalty0 (6):\penalty0 613--616, 1970.

\bibitem[Barakat et~al.(2021)Barakat, Bianchi, Hachem, and
  Schechtman]{barakat_et_al_2021}
A.~Barakat, P.~Bianchi, W.~Hachem, and S.~Schechtman.
\newblock {Stochastic optimization with momentum: Convergence, fluctuations,
  and traps avoidance}.
\newblock \emph{Electron. J. Stat.}, 15\penalty0 (2), 2021.

\bibitem[Bateman(1931)]{bateman_1931}
H.~Bateman.
\newblock On dissipative systems and related variational principles.
\newblock \emph{Phys. Rev.}, 38, 1931.

\bibitem[Bauschke and Combettes(2011)]{bauschke_2011}
H.~Bauschke and P.~Combettes.
\newblock \emph{Convex Analysis and Monotone Operator Theory in Hilbert
  Spaces}.
\newblock Springer, 2011.

\bibitem[Beck(2017)]{beck_2017}
A.~Beck.
\newblock \emph{First-Order Methods in Optimization}.
\newblock SIAM, 2017.

\bibitem[Bena\"im(1996)]{benaim_1996}
M.~Bena\"im.
\newblock A dynamical system approach to stochastic approximations.
\newblock \emph{SIAM J. Control Optim.}, 34\penalty0 (2), 1996.

\bibitem[Bena{\"i}m(1999)]{benaim_1999}
M.~Bena{\"i}m.
\newblock Dynamics of stochastic approximation algorithms.
\newblock In J.~Az{\'e}ma, M.~{\'E}mery, M.~Ledoux, and M.~Yor, editors,
  \emph{S{\'e}minaire de Probabilit{\'e}s XXXIII}. Springer, 1999.

\bibitem[Bengio et~al.(1994)Bengio, Simard, and Frasconi]{bengio_1994}
Y.~Bengio, P.~Simard, and P.~Frasconi.
\newblock Learning long-term dependencies with gradient descent is difficult.
\newblock \emph{IEEE Trans. Neural Netw.}, 5\penalty0 (2), 1994.

\bibitem[Bico et~al.(2009)Bico, Ashmore-Chakrabarty, McKinley, and
  Stone]{bico_et_al_2009}
J.~Bico, J.~Ashmore-Chakrabarty, G.~H. McKinley, and H.~A. Stone.
\newblock {Rolling stones: The motion of a sphere down an inclined plane coated
  with a thin liquid film}.
\newblock \emph{Phys. Fluids}, 21\penalty0 (8):\penalty0 082103, 2009.

\bibitem[Bideau et~al.(1994)Bideau, Riguidel, Hansen, Ristow, Wu, M{\aa}l{\o}y,
  and Ammi]{bideau_et_al_1994}
D.~Bideau, F.~X. Riguidel, A.~Hansen, G.~Ristow, X.~I. Wu, K.~J. M{\aa}l{\o}y,
  and M.~Ammi.
\newblock Granular flow: Some experimental results.
\newblock In K.~K. Bardhan, B.~K. Chakrabarti, and A.~Hansen, editors,
  \emph{Non-Linearity and Breakdown in Soft Condensed Matter}. Springer, 1994.

\bibitem[Billingsley(1968)]{billingsley1968}
P.~Billingsley.
\newblock \emph{Convergence of Probability Measures}.
\newblock Wiley, 2nd edition, 1968.

\bibitem[Borkar(2008)]{borkar_2008}
V.~Borkar.
\newblock \emph{Stochastic Approximation: A Dynamical Systems Viewpoint}.
\newblock Cambridge University Press, 2008.

\bibitem[Bottou et~al.(2018)Bottou, Curtis, and
  Nocedal]{BottouCurtisNocedal.2018}
L.~Bottou, F.~Curtis, and J.~Nocedal.
\newblock Optimization methods for large-scale machine learning.
\newblock \emph{SIAM Rev.}, 60\penalty0 (2):\penalty0 223--311, 2018.

\bibitem[Crawshaw et~al.(2022)Crawshaw, Liu, Orabona, Zhang, and
  Zhuang]{crawshaw2022robustness}
M.~Crawshaw, M.~Liu, F.~Orabona, W.~Zhang, and Z.~Zhuang.
\newblock Robustness to unbounded smoothness of generalized signsgd.
\newblock In S.~Koyejo, S.~Mohamed, A.~Agarwal, D.~Belgrave, K.~Cho, and A.~Oh,
  editors, \emph{Advances in Neural Information Processing Systems}, volume~35,
  pages 9955--9968. Curran Associates, Inc., 2022.

\bibitem[Defazio(2021)]{defazio2021}
A.~Defazio.
\newblock Momentum via primal averaging: Theoretical insights and learning rate
  schedules for non-convex optimization, 2021.
\newblock arXiv: 2010.00406 [cs.LG].

\bibitem[Droniou and Eymard(2016)]{droniou_eymard_2016}
J.~Droniou and R.~Eymard.
\newblock Uniform-in-time convergence of numerical methods for non-linear
  degenerate parabolic equations.
\newblock \emph{Numer. Math.}, 132, 2016.

\bibitem[Ethier and Kurtz(1986)]{ethier_kurtz_1986}
S.~N. Ethier and T.~G. Kurtz.
\newblock \emph{Markov processes -- characterization and convergence}.
\newblock Wiley, 1986.

\bibitem[Fang et~al.(2019)Fang, Lin, and Zhang]{fang_et_al_2019}
C.~Fang, Z.~Lin, and T.~Zhang.
\newblock Sharp analysis for nonconvex {SGD} escaping from saddle points.
\newblock \emph{PMLR}, 99:\penalty0 1192--1234, 2019.

\bibitem[Faw et~al.(2023)Faw, Rout, Caramanis, and Shakkottai]{FawRCS23}
M.~Faw, L.~Rout, C.~Caramanis, and S.~Shakkottai.
\newblock Beyond uniform smoothness: A stopped analysis of adaptive {SGD}.
\newblock In G.~Neu and L.~Rosasco, editors, \emph{COLT}, volume 195 of
  \emph{Proceedings of Machine Learning Research}, pages 89--160. PMLR, 2023.

\bibitem[Fort and Pag\`es(1996)]{FortPages.1996}
J.-C. Fort and G.~Pag\`es.
\newblock Convergence of stochastic algorithms: from the {K}ushner-{C}lark
  theorem to the {L}yapounov functional method.
\newblock \emph{Adv. in Appl. Probab.}, 28\penalty0 (4):\penalty0 1072--1094,
  1996.

\bibitem[Franca et~al.(2020)Franca, Sulam, Robinson, and
  Vidal]{franca_et_al_2020}
G.~Franca, J.~Sulam, D.~Robinson, and R.~Vidal.
\newblock Conformal symplectic and relativistic optimization.
\newblock In H.~Larochelle, M.~Ranzato, R.~Hadsell, M.~Balcan, and H.~Lin,
  editors, \emph{Advances in Neural Information Processing Systems}, volume~33.
  Curran Associates, Inc., 2020.

\bibitem[Franca et~al.(2021)Franca, Jordan, and Vidal]{Franca2020O}
G.~Franca, M.~I. Jordan, and R.~Vidal.
\newblock On dissipative symplectic integration with applications to
  gradient-based optimization.
\newblock \emph{J. Stat. Mech.: Theory Exp.}, 2021.

\bibitem[Freise(2016)]{freise}
F.~Freise.
\newblock \emph{On Convergence of the Maximum Likelihood Estimator in Adaptive
  Designs}.
\newblock PhD thesis, Otto-von-Guericke-Universität Magdeburg, 2016.

\bibitem[Gadat et~al.(2018)Gadat, Panloup, and Saadane]{gadat_et_al_2018}
S.~Gadat, F.~Panloup, and F.~Saadane.
\newblock {Stochastic heavy ball}.
\newblock \emph{Electron. J. Stat.}, 12\penalty0 (1), 2018.

\bibitem[Glendinning(1994)]{glendinning_1994}
P.~Glendinning.
\newblock \emph{Stability, instability and chaos; an introduction to the theory
  of nonlinear differential equations}.
\newblock Cambridge University Press, 1994.

\bibitem[Goldstein et~al.(2014)Goldstein, Poole, and
  Safko]{goldstein_mechanics}
H.~Goldstein, C.~Poole, and J.~Safko.
\newblock \emph{Classical Mechanics}.
\newblock Addison-Wesley, 3rd edition, 2014.

\bibitem[Goodfellow et~al.(2016)Goodfellow, Bengio, and
  Courville]{Goodfellow-et-al-2016}
I.~Goodfellow, Y.~Bengio, and A.~Courville.
\newblock \emph{Deep Learning}.
\newblock MIT Press, 2016.
\newblock URL \url{http://www.deeplearningbook.org}.

\bibitem[Gorbunov et~al.(2020)Gorbunov, Danilova, and
  Gasnikov]{gorbunov2020stochastic}
E.~Gorbunov, M.~Danilova, and A.~Gasnikov.
\newblock Stochastic optimization with heavy-tailed noise via accelerated
  gradient clipping.
\newblock In H.~Larochelle, M.~Ranzato, R.~Hadsell, M.~Balcan, and H.~Lin,
  editors, \emph{Advances in Neural Information Processing Systems}, volume~33,
  pages 15042--15053. Curran Associates, Inc., 2020.

\bibitem[Gurbuzbalaban et~al.(2021)Gurbuzbalaban, Simsekli, and
  Zhu]{gurbuzbalaban21_2021}
M.~Gurbuzbalaban, U.~Simsekli, and L.~Zhu.
\newblock The heavy-tail phenomenon in {SGD}.
\newblock In M.~Meila and T.~Zhang, editors, \emph{Proceedings of the 38th
  International Conference on Machine Learning}, volume 139 of
  \emph{Proceedings of Machine Learning Research}. PMLR, 2021.

\bibitem[Kapoor and Harshvardhan(2021)]{Harshvardhan}
J.~Kapoor and Harshvardhan.
\newblock A stochastic extension of {H}amiltonian descent methods.
\newblock URL \url{https://harshv834.github.io/files/shd_report.pdf}, accessed
  9 January 2024, 2021.

\bibitem[Khaled and Richt{\'a}rik(2023)]{khaled2023bettertheorysgdnonconvex}
A.~Khaled and P.~Richt{\'a}rik.
\newblock Better theory for {SGD} in the nonconvex world.
\newblock \emph{Transactions on Machine Learning Research}, 2023.
\newblock ISSN 2835-8856.

\bibitem[Koloskova et~al.(2023)Koloskova, Hendrikx, and
  Stich]{koloskova_et_al_2023}
A.~Koloskova, H.~Hendrikx, and S.~U. Stich.
\newblock Revisiting gradient clipping: Stochastic bias and tight convergence
  guarantees.
\newblock In A.~Krause, E.~Brunskill, K.~Cho, B.~Engelhardt, S.~Sabato, and
  J.~Scarlett, editors, \emph{International Conference on Machine Learning},
  2023.

\bibitem[Kushner and Yin(2003)]{kushner_yin_2003}
H.~Kushner and G.~Yin.
\newblock \emph{Stochastic Approximation and Recursive Algorithms and
  Applications}.
\newblock Springer, 2nd edition, 2003.

\bibitem[Kushner and Clark(1978)]{kushner_clark_1978}
H.~J. Kushner and D.~S. Clark.
\newblock \emph{Stochastic Approximation Methods for Constrained and
  Unconstrained Systems.}
\newblock Springer, 1st edition, 1978.

\bibitem[LaSalle and Lefschetz(1961)]{LaSalle_Lefschetz_1961}
J.~LaSalle and S.~Lefschetz.
\newblock \emph{Stability by Liapunov's Direct Method with applications}.
\newblock Academic press, 1961.

\bibitem[Leimkuhler and Matthews(2015)]{leimkuhler_2015}
B.~Leimkuhler and C.~Matthews.
\newblock \emph{Molecular Dynamics: With Deterministic and Stochastic Numerical
  Methods}.
\newblock Springer, 2015.

\bibitem[Li et~al.(2023)Li, Rakhlin, and Jadbabaie]{li2023convergence}
H.~Li, A.~Rakhlin, and A.~Jadbabaie.
\newblock Convergence of adam under relaxed assumptions.
\newblock In A.~Oh, T.~Naumann, A.~Globerson, K.~Saenko, M.~Hardt, and
  S.~Levine, editors, \emph{Advances in Neural Information Processing Systems},
  volume~36, pages 52166--52196. Curran Associates, Inc., 2023.

\bibitem[Livingstone et~al.(2017)Livingstone, Faulkner, and
  Roberts]{livingstone_2017}
S.~Livingstone, M.~F. Faulkner, and G.~Roberts.
\newblock Kinetic energy choice in {H}amiltonian/hybrid {M}onte {C}arlo.
\newblock \emph{Biometrika}, 2017.

\bibitem[Ljung(1976)]{ljung_1976}
L.~Ljung.
\newblock \emph{Analysis of Recursive Stochastic Algorithms}.
\newblock Technical Reports TFRT-7097. Department of Automatic Control, Lund
  Institute of Technology (LTH), 1976.

\bibitem[Maddison et~al.(2018)Maddison, Paulin, Teh, O'Donoghue, and
  Doucet]{maddison_et_al_2018}
C.~J. Maddison, D.~Paulin, Y.~W. Teh, B.~O'Donoghue, and A.~Doucet.
\newblock Hamiltonian descent methods, 2018.
\newblock arXiv:1809.05042 [math.OC].

\bibitem[Mai and Johansson(2021)]{mai_johansson_2021}
V.~Mai and M.~Johansson.
\newblock Stability and convergence of stochastic gradient clipping: Beyond
  {L}ipschitz continuity and smoothness.
\newblock In M.~Meila and T.~Zhang, editors, \emph{Proceedings of the 38th
  International Conference on Machine Learning}, volume 139, pages 7325--7335.
  PMLR, 2021.

\bibitem[Mikolov(2013)]{mikolov_2013}
T.~Mikolov.
\newblock \emph{Statistical language models based on neural networks}.
\newblock PhD thesis, Brno University of Technology, 2013.

\bibitem[Nesterov(2018)]{nesterov_2018}
Y.~Nesterov.
\newblock \emph{Lectures on Convex Optimization}.
\newblock Springer, 2018.

\bibitem[Owens and Filkin(1989)]{owens_filkin_1989}
A.~Owens and D.~Filkin.
\newblock Efficient training of the backpropagation network by solving a system
  of stiff ordinary differential equations.
\newblock In \emph{International 1989 Joint Conference on Neural Networks},
  volume~2, pages 381--386, 1989.

\bibitem[Pascanu et~al.(2013)Pascanu, Mikolov, and Bengio]{pascanu_2013}
R.~Pascanu, T.~Mikolov, and Y.~Bengio.
\newblock On the difficulty of training neural networks.
\newblock In S.~Dasgupta and D.~McAllester, editors, \emph{Proceedings of the
  30th International Conference on Machine Learning}, volume 28(3), 2013.

\bibitem[Paszke et~al.(2019)Paszke, Gross, Massa, Lerer, Bradbury, Chanan,
  Killeen, Lin, Gimelshein, Antiga, Desmaison, Kopf, Yang, DeVito, Raison,
  Tejani, Chilamkurthy, Steiner, Fang, Bai, and Chintala]{PyTorch}
A.~Paszke, S.~Gross, F.~Massa, A.~Lerer, J.~Bradbury, G.~Chanan, T.~Killeen,
  Z.~Lin, N.~Gimelshein, L.~Antiga, A.~Desmaison, A.~Kopf, E.~Yang, Z.~DeVito,
  M.~Raison, A.~Tejani, S.~Chilamkurthy, B.~Steiner, L.~Fang, J.~Bai, and
  S.~Chintala.
\newblock Pytorch: An imperative style, high-performance deep learning library.
\newblock In H.~Wallach, H.~Larochelle, A.~Beygelzimer, F.~d\textquotesingle
  Alch\'{e}-Buc, E.~Fox, and R.~Garnett, editors, \emph{Advances in Neural
  Information Processing Systems}, volume~32. Curran Associates, Inc., 2019.

\bibitem[Peressini et~al.(1993)Peressini, Sullivan, and Uhl]{peressini1993}
A.~Peressini, F.~Sullivan, and J.~Uhl, Jr.
\newblock \emph{The Mathematics of Nonlinear Programming}.
\newblock Springer, 1993.

\bibitem[Poljak(1967)]{poljak_1967}
B.~Poljak.
\newblock A general method for solving extremum problems.
\newblock \emph{Sov. Math., Dokl.}, 8\penalty0 (3), 1967.

\bibitem[Polyak(1964)]{polyak_1964}
B.~Polyak.
\newblock Some methods of speeding up the convergence of iteration methods.
\newblock \emph{{USSR} Comp. Math. Math. Phys.}, 4, 1964.

\bibitem[Robbins and Monro(1951)]{robbins_monro_1951}
H.~Robbins and S.~Monro.
\newblock A stochastic approximation algorithm.
\newblock \emph{Ann. Math. Stat.}, 22\penalty0 (3), 1951.

\bibitem[Robbins and Siegmund(1971)]{robbins_siegmund}
H.~Robbins and D.~Siegmund.
\newblock A convergence theorem for nonnegative almost supermartingales and
  some applications.
\newblock In J.~Rustagi, editor, \emph{Optimizing Methods in Statistics}.
  Academic Press, 1971.

\bibitem[Rolski et~al.(2009)Rolski, Schmidli, Schmidt, and
  Teugels]{rolski2009stochastic}
T.~Rolski, H.~Schmidli, V.~Schmidt, and J.~Teugels.
\newblock \emph{Stochastic Processes for Insurance and Finance}.
\newblock Wiley Series in Probability and Statistics. Wiley, 2009.
\newblock ISBN 9780470317884.

\bibitem[Sebbouh et~al.(2021)Sebbouh, Gower, and Defazio]{sebbouh_et_al}
O.~Sebbouh, R.~M. Gower, and A.~Defazio.
\newblock Almost sure convergence rates for stochastic gradient descent and
  stochastic heavy ball.
\newblock In M.~Belkin and S.~Kpotufe, editors, \emph{Proceedings of Thirty
  Fourth Conference on Learning Theory}, volume 134. PMLR, 2021.

\bibitem[Sherman(2013)]{sherman_2013}
J.~Sherman.
\newblock Nearly maximum flows in nearly linear time.
\newblock In \emph{2013 IEEE 54th Annual Symposium on Foundations of Computer
  Science}, pages 263--269, 2013.

\bibitem[Shiryaev(2016)]{shiryaev_2016}
A.~Shiryaev.
\newblock \emph{Probability-1}.
\newblock Graduate Texts in Mathematics. Springer, 3rd edition, 2016.

\bibitem[Stoltz and Trstanova(2018)]{stoltz_trstanova_2018}
G.~Stoltz and Z.~Trstanova.
\newblock Langevin dynamics with general kinetic energies.
\newblock \emph{Multiscale Model. Sim.}, 16\penalty0 (2), 2018.

\bibitem[van~der Vaart(2000)]{van2000asymptotic}
A.~van~der Vaart.
\newblock \emph{Asymptotic Statistics}.
\newblock Cambridge University Press, 2000.

\bibitem[Vladimirova et~al.(2020)Vladimirova, Girard, Nguyen, and
  Arbel]{vladimirova2020sub}
M.~Vladimirova, S.~Girard, H.~Nguyen, and J.~Arbel.
\newblock Sub-{W}eibull distributions: Generalizing sub-{G}aussian and
  sub-{E}xponential properties to heavier-tailed distributions.
\newblock \emph{Stat}, 9\penalty0 (1), 2020.

\bibitem[Wang et~al.(2023{\natexlab{a}})Wang, Zhang, Ma, and Chen]{WangZM023}
B.~Wang, H.~Zhang, Z.~Ma, and W.~Chen.
\newblock Convergence of adagrad for non-convex objectives: Simple proofs and
  relaxed assumptions.
\newblock In G.~Neu and L.~Rosasco, editors, \emph{COLT}, volume 195 of
  \emph{Proceedings of Machine Learning Research}. PMLR, 2023{\natexlab{a}}.

\bibitem[Wang et~al.(2023{\natexlab{b}})Wang, Zhang, Ma, and
  Chen]{pmlr-v195-wang23a}
B.~Wang, H.~Zhang, Z.~Ma, and W.~Chen.
\newblock Convergence of adagrad for non-convex objectives: Simple proofs and
  relaxed assumptions.
\newblock In G.~Neu and L.~Rosasco, editors, \emph{Proceedings of Thirty Sixth
  Conference on Learning Theory}, volume 195 of \emph{Proceedings of Machine
  Learning Research}. PMLR, 2023{\natexlab{b}}.

\bibitem[Williams(1991)]{williams_probability}
D.~Williams.
\newblock \emph{Probability with Martingales.}
\newblock Cambridge University Press, 1991.

\bibitem[Wolf et~al.(1998)Wolf, Radjai, and Dippel]{wolf_et_al_1998}
D.~Wolf, F.~Radjai, and S.~Dippel.
\newblock Dissipation in granular materials.
\newblock \emph{Philos. Mag. B}, 77\penalty0 (5):\penalty0 1413--1425, 1998.

\bibitem[Yoshizawa(1959)]{Yoshizawa_1959}
T.~Yoshizawa.
\newblock Liapunov’s function and boundedness of solutions.
\newblock \emph{Funkcialaj Ekvacioj}, 2, 1959.

\bibitem[Yoshizawa(1966)]{Yoshizawa_1966}
T.~Yoshizawa.
\newblock \emph{Stability Theory by Liapunov’s second method}.
\newblock The mathematical society of Japan, 1966.

\bibitem[Zeidler(1985)]{zeidler3_1985}
E.~Zeidler.
\newblock \emph{Nonlinear Functional Analysis and Its Applications: III:
  Variational Methods and Optimization}.
\newblock Springer, 1985.

\bibitem[Zhang et~al.(2020{\natexlab{a}})Zhang, Jin, Fang, and
  Wang]{zhang_2020}
B.~Zhang, J.~Jin, C.~Fang, and L.~Wang.
\newblock Improved analysis of clipping algorithms for non-convex optimization.
\newblock In H.~Larochelle, M.~Ranzato, R.~Hadsell, M.~Balcan, and H.~Lin,
  editors, \emph{Advances in Neural Information Processing Systems}, volume~33,
  pages 15511--15521. Curran Associates, Inc., 2020{\natexlab{a}}.

\bibitem[Zhang et~al.(2020{\natexlab{b}})Zhang, He, Sra, and
  Jadbabaie]{Zhang2020Why}
J.~Zhang, T.~He, S.~Sra, and A.~Jadbabaie.
\newblock Why gradient clipping accelerates training: A theoretical
  justification for adaptivity.
\newblock In \emph{International Conference on Learning Representations},
  2020{\natexlab{b}}.

\end{thebibliography}
\bibliographystyle{abbrvnat}

\appendix

\section{Analysis}\label{appendix:analysis}
As explained in Section~\ref{sec:outline_of_proof}, the two main steps of the convergence analysis are to first prove that $p_k$ and $q_k$ are finite almost surely, and then to use this a priori result to show that they in fact converge.

\subsection{The sequences $\{p_k\}_{k \geq 0}$ and $\{q_k\}_{k \geq 0}$ are finite almost surely}\label{appendix:pkqk_bdd_as}
We first prove Theorem \ref{thm:pqbounded}:
\thmpkqkbdd*
The proof relies on the Robbins-Siegmund theorem:
\begin{theorem}[\citep{robbins_siegmund}]\label{thm:robbins_siegmund}
Let $(\Omega,\mathcal{F}, \Prob)$ be a probability space and 
$\mathcal{F}_1 \subset \mathcal{F}_2 \subset  \dots$ be a sequence of sub-$\sigma$-algebras of $\mathcal{F}$.
For each $k=1,2,\dots$ let $V_k,\beta_k,X_k$ and $Y_k$ be 
non-negative $\mathcal{F}_k$-measurable random variables such that
\begin{align*}
\E \left[ V_{k+1} | \mathcal{F}_k \right]
\leq
V_k(1 + \beta_k) + X_k - Y_k.
\end{align*}
Then on the set
$
\left\{ \omega: \sum_k \beta_k < \infty, \sum_k X_k <\infty \right\},
$
the limit
\begin{align*}
\lim_{k \to \infty} V_k = V 
\end{align*}
exists and is finite and $\sum_k Y_k < \infty$.
\end{theorem}
We first consider Setting 1, i.e.\ with Assumption~\ref{ass:setting_one}.
The strategy is to introduce  $V_k = H(p_k,q_k) -F_* - \varphi_*= F(q_k) - F_*+ \varphi(p_k)- \varphi_*$, then use $L$-smoothness of $F$ and the convexity of $\varphi$ to bound the difference $V_{k+1} -V_k$ by 
$\alpha_k^2 V_k$ plus higher-order terms of $\alpha_k$. Then we can appeal to Theorem \ref{thm:robbins_siegmund} to conclude that $\{p_k\}_{k\geq1}$ and $\{q_k\}_{k\geq1}$ are finite a.s.

\begin{proof}[Proof of Theorem \ref{thm:pqbounded} in Setting 1]
Let $V_k = H(p_k,q_k) - F_* - \varphi_*$.
Then we have that
\begin{align}\label{eq:Vk_diff}
V_{k+1} - V_k
=  F(q_{k+1}) - F(q_k)
+\varphi(p_{k+1}) - \varphi(p_k).
\end{align}
By $L$-smoothness of $F$ and convexity of $\varphi$, this is less than or equal to
\begin{align*}
\langle \nabla F(q_k), q_{k+1} -q_k \rangle + \frac{L}{2} \lVert q_{k+1} -q_k \rVert_2^2 + \langle \nabla \varphi(p_{k+1}),p_{k+1} - p_k \rangle.
\end{align*}
We insert (\ref{eq:stochastic_Hamiltonian_descent}) into the previous expression to obtain that it is equal to
\begin{align*}
  I_1 + I_2 + I_3 :&= \alpha_k \langle \nabla F(q_k)-\nabla f(q_k,\xi_k) , \nabla \varphi(p_{k+1}) \rangle \\
  &\quad+ \frac{L\alpha_k ^2}{2} \lVert  \nabla \varphi(p_{k+1})\rVert_2^2 
  -
                 \alpha_k \gamma \langle \nabla \varphi(p_{k+1}), \nabla \varphi(p_k) \rangle,
\end{align*}
where
\begin{align*}
I_1 &= 
\alpha_k \langle \nabla F(q_k)-\nabla f(q_k,\xi_k) , \nabla \varphi(p_{k}) \rangle \\
&+
\alpha_k \langle \nabla F(q_k)-\nabla f(q_k,\xi_k) , \nabla \varphi(p_{k+1})-  
\nabla \varphi(p_{k}) 
\rangle.
\end{align*}
When we take the conditional expectation (w.r.t. the sigma algebra generated by $\xi_1,\dots,\xi_{k-1}$) of $I_1$, the first term is $0$ by the unbiasedness of the gradient and the independence of $\{\xi_k\}$:
\begin{align*}
\mathbb{E}_{\xi_k} \left[I_1 \right]= 
\alpha_k \mathbb{E}_{\xi_k} \left[ \langle \nabla F(q_k)-\nabla f(q_k,\xi_k) , \nabla \varphi(p_{k+1})-  
\nabla \varphi(p_{k}) 
\rangle  \right] .
\end{align*}
Using Assumption \ref{ass:kinetic_energy_function}.\ref{ass:Lipschitzcontinuous_gradient_phi}, we can bound $I_3$ as
\begin{align*}
I_3=
-
 \alpha_k \gamma \langle \nabla \varphi(p_{k+1}),\nabla \varphi( p_k )\rangle
 &\leq \frac{\alpha_k \gamma}{2}
 \lVert \nabla \varphi(p_{k+1}) - \nabla \varphi(p_k) \rVert_2^2 \\
 &\leq  
 \frac{ \alpha_k \gamma \lambda^2}{2}
 \lVert p_{k+1} - p_k \rVert_2^2.
\end{align*}
After taking the expectation of (\ref{eq:Vk_diff}), we thus get 
the bound
\begin{align}\label{eq:Vk_diff2}
\begin{split}
\mathbb{E}_{\xi_k} \left[ V_{k+1} \right] - V_k
&\leq 
\alpha_k \mathbb{E}_{\xi_k} \left[ \langle \nabla F(q_k)-\nabla f(q_k,\xi_k) , \nabla \varphi(p_{k+1})-  
\nabla \varphi(p_{k}) 
\rangle  \right] \\  & \quad  + 
\frac{L\alpha_k ^2}{2} \mathbb{E}_{\xi_k} \left[ \lVert  \nabla \varphi(p_{k+1})\rVert_2^2 \right]
+ \frac{ \alpha_k \gamma \lambda^2}{2}
 \mathbb{E}_{\xi_k} \left[\lVert p_{k+1} - p_k \rVert_2^2 \right]
 \\&=: I_1' + I_2 + I_3'.
 \end{split}
\end{align}
We now make use of Cauchy--Schwarz inequality along with the Lipschitz continuity of $\nabla \varphi$ to bound $I_1'$ as 
\begin{align*}
I_1' & \leq 
\alpha_k \lambda \mathbb{E}_{\xi_k} \left[ \lVert \nabla F(q_k)-\nabla f(q_k,\xi_k) \rVert_2 \lVert p_{k+1} -p_k
\rVert_2 \right].
\end{align*}
We insert (\ref{eq:stochastic_Hamiltonian_descent}) into the previous expression, and make use of Young's inequality for products, $ab \leq \frac{a^2}{2} + \frac{b^2}{2}$, to obtain that
\begin{align*}
I_1' &\leq
\alpha_k^2  \lambda \mathbb{E}_{\xi_k} \left[ \lVert \nabla F(q_k)-\nabla f(q_k,\xi_k) \rVert_2 \lVert \nabla f(q_k,\xi_k) + \gamma \nabla \varphi(p_k)
\rVert_2 \right] \\
& \leq
\frac{\alpha_k^2}{2}  \lambda \mathbb{E}_{\xi_k} \left[ \lVert \nabla F(q_k)-\nabla f(q_k,\xi_k) \rVert_2^2 \right]\\
&\quad+
\frac{\alpha_k^2}{2}  \lambda \mathbb{E}_{\xi_k} \left[  \lVert \nabla f(q_k,\xi_k) + \gamma \nabla \varphi(p_k)
\rVert_2^2 \right].
\end{align*}
Making use of the inequality
\begin{align}\label{eq:norm_squared_ineq}
\lVert x-y\rVert_2^2 \leq 2\lVert x\rVert_2^2 + 2\lVert y \rVert_2^2,
\end{align}
we can further bound $I_1'$ by
\begin{align*}
I_1'& \leq  \frac{\alpha_k^2}{2}  \lambda \mathbb{E}_{\xi_k} \left[ \lVert \nabla F(q_k)-\nabla f(q_k,\xi_k) \rVert_2^2 \right]
\\
&\quad+
\alpha_k^2  \lambda \mathbb{E}_{\xi_k} \left[  \lVert \nabla f(q_k,\xi_k)
\rVert_2^2 \right]
+
\alpha_k^2  \lambda \gamma ^2 \mathbb{E}_{\xi_k} \left[  \lVert  \nabla \varphi(p_k)
\rVert_2^2 \right].
\end{align*}
At last we make use of Assumption \ref{ass:setting_one}.\ref{ass:expectedsmoothness} to get that
\begin{align*}
I_1'& \leq  \frac{\alpha_k^2}{2}  \lambda \left( \kappa (F(q_k) - F_*) + \tau \lVert \nabla F(q_k) \rVert_2^2 + \sigma^2
\right)
\\
&\quad+
\alpha_k^2  \lambda \left( \kappa (F(q_k) - F_*) + (1+\tau) \lVert \nabla F(q_k) \rVert_2^2 + \sigma^2
\right) \\
&\quad+
\alpha_k^2  \lambda \gamma ^2 \mathbb{E}_{\xi_k} \left[  \lVert  \nabla \varphi(p_k)
\rVert_2^2 \right].
\end{align*}
We now turn our attention to the term $I_2$ in (\ref{eq:Vk_diff2}). Adding and subtracting $\nabla \varphi(p_k)$ and making use of Assumption \ref{ass:kinetic_energy_function}.\ref{ass:Lipschitzcontinuous_gradient_phi} we get that
\begin{align*}
I_2 &\leq
\frac{L\alpha_k ^2}{2}\mathbb{E}_{\xi_k} \left[ \lVert  \nabla \varphi(p_{k+1}) - \nabla \varphi(p_{k})  \rVert_2^2 \right]
+ 
\frac{L\alpha_k ^2}{2}\mathbb{E}_{\xi_k} \left[ \lVert \nabla \varphi(p_{k}) \rVert_2^2 \right]
\\
&\leq 
\frac{L\lambda^2  \alpha_k ^2}{2}\mathbb{E}_{\xi_k} \left[ \lVert   p_{k+1} - p_{k}  \rVert_2^2 \right]
+ 
\frac{L\alpha_k ^2}{2} \lVert \nabla \varphi(p_{k}) \rVert_2^2
\\ & \leq
L\alpha_k ^4 \lambda^2 \mathbb{E}_{\xi_k} \left[ \lVert   \nabla f(q_k,\xi_k)   \rVert_2^2 \right]
+
\left(L\alpha_k ^4 \gamma ^2 \lambda^2 
+ 
\frac{L\alpha_k ^2}{2} \right) \lVert \nabla \varphi(p_{k}) \rVert_2^2,
\end{align*}
where we have used (\ref{eq:stochastic_Hamiltonian_descent}) and (\ref{eq:norm_squared_ineq}) in the last step.
Making use of Assumption \ref{ass:kinetic_energy_function}.\ref{ass:Lipschitzcontinuous_gradient_phi} again we obtain that
\begin{align*}
I_2&\leq
L\alpha_k ^4 \lambda^2 
 \left( \kappa (F(q_k) - F_*) + (1+\tau) \lVert \nabla F(q_k) \rVert_2^2 + \sigma^2
\right) \\
 &\quad+
\left(L\alpha_k ^3 \gamma ^2 \lambda^2 
+ 
\frac{L\alpha_k ^2}{2} \right) \lVert \nabla \varphi(p_{k}) \rVert_2^2.
\end{align*}
In a similar way, we find that
\begin{align*}
I_3' \leq 
\alpha_k^3 \gamma \lambda^2 
\left( \kappa (F(q_k) - F_*) + (1+\tau) \lVert \nabla F(q_k) \rVert_2^2 + \sigma^2
\right)
+ \alpha_k^3 \gamma^3 \lambda^2 \lVert \nabla \varphi(p_k) \rVert_2^2.
\end{align*}
Gathering up the terms, we get that
\begin{align*}
\begin{split}
  &\E_{\xi_k} [V_{k+1}] - V_k \\
  &\leq \left(\frac{\alpha_k^2  \lambda}{2}+\alpha_k^2  \lambda +L\alpha_k ^4 \lambda^2 +  \alpha_k^3 \gamma \lambda^2  \right) \sigma^2 
\\
& \quad +
\alpha_k^2 \kappa \lambda \left( \frac{3}{2} + L \alpha_k^2 \lambda + \alpha_k \gamma \lambda\right) 
(F(q_k) - F_*)
\\
& \quad + \biggl(\frac{\alpha_k^2  \lambda \tau}{2} +\alpha_k^2  \lambda(1 + \tau) +L\alpha_k ^4 \lambda^2 (1+ \tau) 
  +  \alpha_k^3 \gamma \lambda^2(1 + \tau)  \biggr) \lVert \nabla F(q_k)
\rVert_2^2
\\
& \quad + \left(
\alpha_k^2  \gamma ^2  \lambda 
+
L\alpha_k ^4 \lambda^2  \gamma^2
+ 
\frac{L\alpha_k ^2}{2}
+
\alpha_k^3 \gamma^3 \lambda^2  \right)\lVert \nabla \varphi(p_k) \rVert_2^2
 .
\end{split}
\end{align*}
Making use of Lemma \ref{lemma:smoothbound}, we see that
\begin{align*}
  &\E_{\xi_k}[ V_{k+1} ] - V_k \\
  &\leq \biggl(\frac{\alpha_k^2  \lambda}{2}+\alpha_k^2  \lambda +L\alpha_k ^4 \lambda^2 +  \alpha_k^3 \gamma \lambda^2  \biggr) \sigma^2 
\\
&\qquad +
(F(q_k) - F_*) \Biggl[ \alpha_k^2 \kappa \lambda \biggl( \frac{3}{2} + L \alpha_k^2 \lambda + \alpha_k \gamma \lambda\biggr) 
\\
&\qquad\quad + 2L  \biggl(\frac{\alpha_k^2  \lambda \tau}{2} +\alpha_k^2  \lambda(1 + \tau) +L\alpha_k ^4 \lambda^2 (1+ \tau) +  \alpha_k^3 \gamma \lambda^2(1 + \tau)  \biggr) \Biggr] 
\\
&\qquad + \biggl(
\alpha_k^2  \gamma ^2  \lambda 
+
L\alpha_k ^4 \lambda^2  \gamma^2
+ 
\frac{L\alpha_k ^2}{2}
+
\alpha_k^3 \gamma^3 \lambda^2  \biggr)(\varphi(p_k) - \varphi_*)
 .
\end{align*}
Now define
\begin{align*}
C_1(\alpha_k) = \sigma^2 \left(\frac{\alpha_k^2  \lambda}{2}+\alpha_k^2  \lambda +L\alpha_k ^4 \lambda^2 +  \alpha_k^3 \gamma \lambda^2  \right)
\end{align*}
and let $C_2(\alpha_k)$ be the maximum of the terms in front of $F(q_k) - F_*$ and $\varphi(p_k) - \varphi_*$.
It follows that
\begin{align}\label{eq:robbinssiegmundbound}
\mathbb{E}_{\xi_k}\left[ V_{k+1}\right] - V_k \leq
C_1(\alpha_k) +  C_2(\alpha_k)  V_k .
\end{align}
Since $C_1$ and $C_2$ only contain second-order terms of $\alpha_k$ (and by assumption $\sum_{k=1}^{\infty}\alpha_k^2 < \infty$), we have that
\begin{align*}
\sum_{k=0}^{\infty} C_1(\alpha_k) < \infty, \ \sum_{k=0}^{\infty} C_2(\alpha_k) < \infty.
\end{align*}
We can thus make use of the Robbins--Siegmund theorem with $\beta_k = C_2(\alpha_k)$, $X_k = C_1(\alpha_k)$ and $Y_k =0$ to conclude that $V_k$ tends to a non-negative, finite, random variable $V$ almost surely.
Since $F$ and $\varphi$ are assumed to be coercive, this implies that $\{p_k\}_{k\geq1}$ and $\{q_k\}_{k\geq1}$ are finite almost surely.
For the second claim of the proof, we define 
\begin{align*}
S_k = \frac{V_k}{\prod_{j=0}^{k-1}(1 +C_2(\alpha_j) )}.
\end{align*}
By (\ref{eq:robbinssiegmundbound}), we have that
\begin{align*}
\mathbb{E}_{\xi_k}\left[ S_{k+1}\right]  \leq
S_k + 
\frac{C_1(\alpha_k)}{\prod_{j=0}^{k}(1 + C_2(\alpha_j) )}
\leq 
S_k  +C_1(\alpha_k).
\end{align*}
Taking the expectation and, summing from $0$ to $K-1$, we see that
\begin{align*}
\mathbb{E}\left[ S_{K}\right]  \leq
S_0 + \sum_{k=0}^{K-1} C_1(\alpha_k).
\end{align*}
We multiply both sides of the previous inequality with $\prod_{k=0}^{K-1}(1 + C_2(\alpha_k) )$ 
\begin{align*}
\mathbb{E}\left[ V_{K}\right]  &\leq
\prod_{k=0}^{K-1}(1 + C_2(\alpha_k) )
\left(S_0 + \sum_{k=0}^{K-1} C_1(\alpha_k)\right)
\\
&\leq
e^{ \sum_{j=0}^{K-1}C_2(\alpha_k) )} \cdot
\left(S_0 + \sum_{k=0}^{K-1} C_1(\alpha_k)\right),
\end{align*}
where we have used the fact that $1+x \leq e^{ x}$ in the second step.
Letting $K$ tend to infinity on the left hand side and using the fact that $\sum_{k=0}^{\infty}C_2(\alpha_k) ) < \infty$ we see that the last claim of the theorem also holds:
\begin{align*}
\sup_{k} \mathbb{E}\left[ V_k\right] < \infty.
\end{align*}
\end{proof}
We now give a proof of Theorem~\ref{thm:pqbounded} in Setting 2, i.e.\ with Assumption~\ref{ass:setting_two}.
\begin{proof}[Proof of Theorem \ref{thm:pqbounded} in Setting 2]
By Assumption \ref{ass:setting_two}.\ref{ass:L0L1smoothness}
it holds that
\begin{align}\label{ineq:L0L10}
F(x_{k+1}) - F(x_k) \leq \langle \nabla F(x_k), x_{k+1} - x_k \rangle + 
\frac{L_0 + L_1 \lVert \nabla F(x_k) \rVert}{2} \lVert x_{k+1} - x_k \rVert^2
\end{align}
if $\lVert x_k - x_{k+1} \rVert \leq \frac{1}{L_1}$.
Since
\begin{align}\label{eq:qk}
q_{k+1} - q_k = \alpha_k \nabla \varphi(p_{k+1})
\end{align}
we get for large enough $k$ that
\begin{align}\label{ineq:boundnormdiffqk}
\lVert q_{k+1} - q_k \rVert = \alpha_k \lVert \nabla \varphi(p_{k+1}) \rVert
\leq \alpha_k  \Delta \leq \frac{1}{L_1},
\end{align}
by Assumption~\ref{ass:setting_two}.\ref{ass:bounded_clipping}.
If we insert (\ref{eq:qk}) into (\ref{ineq:L0L10}), we get
\begin{align*}
  F(q_{k+1}) - F(q_k) &\leq \langle \nabla F(q_k), q_{k+1} - q_k \rangle \\
                      &\quad+ 
\frac{L_0}{2} \lVert q_{k+1} - q_k \rVert^2
+
\frac{ L_1}{2}
\lVert \nabla F(q_k) \rVert \lVert q_{k+1} - q_k \rVert^2.
\end{align*}
By (\ref{ineq:boundnormdiffqk}), we get that
\begin{align*}
F(q_{k+1}) - F(q_k) \leq  \alpha_k \langle \nabla F(q_k), \nabla \varphi(p_{k+1}) \rangle + 
\frac{L_0}{2} \alpha_k^2   \Delta^2
+
\frac{ L_1}{2}
\lVert \nabla F(q_k) \rVert \alpha_k^2  \Delta^2.
\end{align*}
By Assumption \ref{ass:kinetic_energy_function}.\ref{ass:convexity} we have that
\begin{align*}
\varphi(p_{k+1}) - \varphi(p_k) &\leq \langle \nabla \varphi(p_{k+1}), p_{k+1}-p_k \rangle \\
&=
-
\alpha_k \langle \nabla \varphi(p_{k+1}), \nabla f(q_k,\xi_k) \rangle
-
\alpha_k \gamma \langle \nabla \varphi(p_{k+1}), \nabla \varphi(p_k) \rangle
\end{align*}
With $H(p,q) = F(q) + \varphi(p)$ and $V_k = H(p_k,q_k) - F_* - \varphi_*$
as in the previous proof we thus get that
\begin{align*}
V_{k+1} - V_k &\le \alpha_k \langle \nabla F(q_k), \nabla \varphi(p_{k+1}) \rangle + 
\frac{L_0}{2} \alpha_k^2  \Delta^2
+
\frac{\alpha_k^2 L_1}{2}
\lVert \nabla F(q_k) \rVert \Delta^2
\\
&-
\alpha_k \langle \nabla \varphi(p_{k+1}), \nabla f(q_k,\xi_k) \rangle
-
\alpha_k \gamma \langle \nabla \varphi(p_{k+1}), \nabla \varphi(p_k) \rangle,
\end{align*}
which can be rewritten as
\begin{align*}
  V_{k+1} - V_k &\le \alpha_k \langle \nabla F(q_k)- \nabla f(q_k,\xi_k), \nabla \varphi(p_{k+1}) \rangle \\
  &\quad+ 
\frac{L_0}{2} \alpha_k^2  \Delta^2
+
\frac{\alpha_k^2 L_1}{2}
\lVert \nabla F(q_k) \rVert \Delta^2
    -
\alpha_k \gamma \langle \nabla \varphi(p_{k+1}), \nabla \varphi(p_k) \rangle.
\end{align*}
We add and subtract $\nabla \varphi(p_k)$ in the first scalar product:
\begin{align}\label{lyapunovbound}
\begin{split}
  V_{k+1} - V_k &\le \alpha_k \langle \nabla F(q_k)- \nabla f(q_k,\xi_k), \nabla \varphi(p_{k+1}) - \nabla \varphi(p_k)\rangle \\
  &\quad+ 
\alpha_k \langle \nabla F(q_k)- \nabla f(q_k,\xi_k),   \nabla \varphi(p_k)\rangle
\\
&\quad +
\frac{L_0}{2} \alpha_k^2  \Delta^2
+
\frac{\alpha_k^2 L_1}{2}
\lVert \nabla F(q_k) \rVert \Delta^2
-
\alpha_k \gamma \langle \nabla \varphi(p_{k+1}), \nabla \varphi(p_k) \rangle.
\end{split}
\end{align}
The second scalar product disappears due to the unbiasedness of $\nabla f(q_k,\xi_k)$ and the fact that $\xi_k$ is independent of $q_k$ and $p_k$. 
We now focus on the first scalar product in (\ref{lyapunovbound}). Taking the conditional expectation and using Cauchy--Schwarz inequality, we see that 
\begin{align*}
&\alpha_k \mathbb{E}_{\xi_k}\left[  \langle \nabla F(q_k)- \nabla f(q_k,\xi_k), \nabla \varphi(p_{k+1}) - \nabla \varphi(p_k)\rangle  \right]
\\
& \qquad
\leq 
\alpha_k \mathbb{E}_{\xi_k}\left[ \lVert \nabla F(q_k)- \nabla f(q_k,\xi_k)\rVert \lVert  \nabla \varphi(p_{k+1}) - \nabla \varphi(p_k)\rVert \right]
\end{align*}
Using the Lipschitz continuity of $\nabla \varphi$, this can be further bounded by
\begin{align}\label{eq:boundtobereusedinsetting3}
\begin{split}
&\alpha_k^2  \lambda  \mathbb{E}_{\xi_k}\left[ \lVert \nabla F(q_k)- \nabla f(q_k,\xi_k)\rVert\lVert  \nabla f(q_k,\xi_k) - \gamma \nabla \varphi(p_k) \rVert \right]
\\
& \qquad \leq
\alpha_k^2 \lambda \mathbb{E}_{\xi_k}\left[ \lVert \nabla F(q_k)- \nabla f(q_k,\xi_k)\rVert  \left(\lVert  \nabla f(q_k,\xi_k) \rVert + \gamma \lVert \nabla \varphi(p_k) \rVert \right) \right].
\end{split}
\end{align}
We now add and subtract $\nabla F(q_k)$ inside the  $\lVert  \nabla f(q_k,\xi_k) \rVert$-term and make use of the triangle inequality to bound the previous expression by
\begin{align*}
&
\alpha_k^2 \lambda \mathbb{E}_{\xi_k}\left[ \lVert \nabla F(q_k)- \nabla f(q_k,\xi_k)\rVert  \left(\lVert   \nabla F(q_k) - \nabla f(q_k,\xi_k) \rVert + \lVert \nabla F(q_k) \rVert + \gamma \Delta \right) \right]
\\
& \quad
=
\alpha_k^2 \lambda \mathbb{E}_{\xi_k}\left[ \lVert \nabla F(q_k)- \nabla f(q_k,\xi_k)\rVert^2   + \lVert   \nabla F(q_k) - \nabla f(q_k,\xi_k) \rVert \lVert \nabla F(q_k) \rVert \right.
\\
 & \qquad + \left. \gamma \Delta \lVert   \nabla F(q_k) - \nabla f(q_k,\xi_k) \rVert  \right]
\end{align*}
where we have also used the assumption that $\lVert \nabla \varphi(p_k) \rVert \leq \Delta$.
Now, the first term can by Assumption \ref{ass:setting_two}.\ref{ass:variance} be bounded by
\begin{align*}
\mathbb{E}_{\xi_k}\left[ \lVert \nabla F(q_k)- \nabla f(q_k,\xi_k)\rVert^2  \right] \leq \sigma^2.
\end{align*}
By Remark \ref{remark:lyapunov}, we can bound the second term as follows
\begin{align*}
\mathbb{E}_{\xi_k}\left[  \lVert   \nabla F(q_k) - \nabla f(q_k,\xi_k) \rVert \lVert \nabla F(q_k) \rVert \right] \leq
\sigma \lVert \nabla F(q_k) \rVert,
\end{align*}
since $\xi_k$ is independent of $\lVert \nabla F(q_k) \rVert$. Likewise, we can bound the last expectation by $\sigma$. Thus, we arrive at the bound
\begin{equation}\label{boundfirstinnerproduct}
\begin{aligned}
\alpha_k &\mathbb{E}_{\xi_k}\left[  \langle \nabla F(q_k)- \nabla f(q_k,\xi_k), \nabla \varphi(p_{k+1}) - \nabla \varphi(p_k)\rangle  \right]\\
&\leq 
\alpha_k^2 \lambda \left(
\sigma^2 + \sigma \lVert \nabla F(q_k) \rVert + \gamma \Delta \sigma
\right).
\end{aligned}
\end{equation}

We can bound the last inner product of (\ref{lyapunovbound}) using Lemma B.1 in \cite{zhang_2020} with $\mu = 0$:
\begin{align*}
&-\alpha_k \gamma \langle \nabla \varphi(p_{k+1}), \nabla \varphi(p_k) \rangle \\
&\quad\leq
- \alpha_k \gamma \lVert \nabla \varphi(p_k) \rVert^2 
+ 
\alpha_k \gamma \lVert \nabla \varphi(p_{k+1}) - \nabla \varphi(p_k) \rVert \lVert \nabla \varphi(p_k) \rVert
\\
&\quad\leq
\alpha_k^2 \gamma 
\lambda 
2 \lVert \nabla f(q_k,\xi_k) - \gamma \nabla \varphi(p_k) \rVert \Delta.
\end{align*}
By Remark \ref{remark:lyapunov} we thus get 
\begin{align}\label{boundlastinnerproduct}
\begin{split}
&-\alpha_k \gamma \mathbb{E}_{\xi_k}\left[ \langle \nabla \varphi(p_{k+1}), \nabla \varphi(p_k) \rangle \right] \\
&\qquad \leq
\alpha_k^2 \gamma 
\lambda 
 \Delta
 \left( \mathbb{E}_{\xi_k}\left[ \lVert \nabla f(q_k,\xi_k) \rVert \right] + \gamma \lVert \nabla \varphi(p_k) \rVert
\right)
\\
&\qquad
\leq
\alpha_k^2 \gamma \lambda  \Delta
\big(
\mathbb{E}_{\xi_k}\left[ \lVert \nabla f(q_k,\xi_k) - \nabla F(q_k) \rVert \right]
+ \lVert \nabla F(q_k) \rVert_2 + \gamma \lVert \nabla \varphi(p_k) \rVert \big)
\\
&\qquad 
\leq
\alpha_k^2 \gamma 
\lambda 
 \Delta
\left( \sigma + \lVert \nabla F(q_k) \rVert_2 + \gamma \Delta
\right).
\end{split}
\end{align}
Inserting (\ref{boundfirstinnerproduct}) and (\ref{boundlastinnerproduct})
into (\ref{lyapunovbound}), we get that
\begin{align*}
  \mathbb{E}_{\xi_k}\left[ V_{k+1} \right] - V_k
  &\le
\alpha_k^2 \lambda \left(
\sigma^2 + \sigma \lVert \nabla F(q_k) \rVert + \gamma \Delta \sigma
\right)+
\frac{L_0}{2} \alpha_k^2  \Delta^2
\\
&\quad+
\frac{\alpha_k^2 L_1}{2}
\lVert \nabla F(q_k) \rVert \Delta^2
+
\alpha_k^2 \gamma 
\lambda 
 \Delta
\left( \sigma + \lVert \nabla F(q_k) \rVert_2 + \gamma \Delta
\right)
\end{align*}
By Lemma \ref{lemma:smoothbound}, we can bound the $\lVert \nabla F(q_k) \rVert-$terms, and obtain the bound
\begin{align*}
  \mathbb{E}_{\xi_k}\left[ V_{k+1} \right] - V_k
  &\le
\alpha_k^2 \lambda \left(
\sigma^2 + \sigma \left( 2 L_1(F(q) - F_*)  +    \frac{ L_0}{L_1} \right) + \gamma \Delta \sigma
\right)+
\frac{L_0}{2} \alpha_k^2  \Delta^2
\\
&\quad+
\frac{\alpha_k^2 L_1}{2}
\left( 2 L_1(F(q) - F_*)  +    \frac{ L_0}{L_1} \right)  \Delta^2
\\
&\quad+
\alpha_k^2 \gamma 
\lambda 
 \Delta
\left( \sigma +\left( 2 L_1(F(q) - F_*)  +    \frac{ L_0}{L_1} \right) + \gamma \Delta
\right).
\end{align*}
We now define
\begin{align*}
C_1(\alpha_k) &=
2\alpha_k^2 \lambda \sigma L_1 + \alpha_k^2 L_1^2 \Delta^2
+
2\alpha_k^2 \gamma \lambda \Delta L_1,\\
C_2(\alpha_k) &= \alpha_k^2 \lambda \sigma^2 
+ \alpha_k^2 \lambda \sigma \frac{L_0}{L_1}
+ \alpha_k^2 \lambda \gamma \Delta \sigma
+ \alpha_k^2 L_0 \Delta^2 \\
&\quad
+ \alpha_k^2 \gamma \lambda \Delta \sigma
+ \alpha_k^2 \gamma \lambda \Delta \frac{L_0}{L_1}
+ \alpha_k^2 \gamma^2 \lambda \Delta^2 
.
\end{align*}
We see that
\begin{align}\label{eq:Hksettingtwo}
\begin{split}
 \mathbb{E}_{\xi_k}\left[ V_{k+1} \right] - V_k
 &\leq 
 C_1 (\alpha_k) (F(q_k) - F_*) + C_2( \alpha_k)
 \\
 &\leq
  C_1 (\alpha_k) (H(p_k,q_k) - F_* - \varphi_*) + C_2( \alpha_k),
  \end{split}
\end{align}
where we have used the fact that $\varphi(p_k) - \varphi_* \geq 0$.
Since $\sum_{k \geq 0} C_i(\alpha_k) < \infty$ for $i=1,2$, we can appeal to the Robbins--Siegmund theorem to conclude that $\lim_{k \to \infty} V_k$ exists and is finite almost surely. Since $F$ and $\varphi$ are coercive this implies that $\sup_k \lVert p_k \rVert < \infty$ and $\sup_k \lVert q_k \rVert < \infty$ almost surely.

\end{proof}

\begin{proof}[Proof of Theorem \ref{thm:pqbounded} in Setting 3]
Let $F_*$ be as in Lemma \ref{lemma:asboundstochasticgradient}. By Lemma \ref{lemma:asboundstochasticgradient}, we have that the right-hand side of (\ref{eq:boundtobereusedinsetting3}) can be bounded by
\begin{align*}
\alpha_k^2 \lambda \mathbb{E}_{\xi_k}\left[ \lVert \nabla F(q_k)- \nabla f(q_k,\xi_k)\rVert  \right] \left(2L_1 N (F(q_k) - F_*) + \gamma \Delta \right).
\end{align*}
By Assumption \ref{ass:setting_three}.\ref{ass:heavytailednoise} this can in its turn be bounded by
\begin{align*}
\alpha_k^2 \lambda \sigma \left(2L_1 N (F(q_k) - F_*) + \gamma \Delta \right).
\end{align*}
The rest of the proof proceeds exactly like that of Setting 2 (with suitable modifications of the constants in the bound (\ref{eq:Hksettingtwo})).
Similarly to the proof in Setting 1, we can also define 
\begin{align*}
S_k = \frac{V_k}{\prod_{j=0}^{k-1}(1 +C_2(\alpha_j) )}
\end{align*}
and in a similar fashion obtain that $\sup_k \mathbb{E}\left[ F(q_k) \right] < + \infty$.
\end{proof}

\subsection{Almost sure convergence, notation}
To prove convergence, we start by rewriting the processes (\ref{def:PkQk}) on a form that is more reminiscent of the integral equations (\ref{eq:nearly_Hamiltonian_system_integral}).
As in \citet{kushner_yin_2003}, we use the convention that
\begin{align*}
\begin{split}
\sum_{i=n}^k a_i&= 0, \ \text{ if } k=n-1 \text{ (the empty sum)},\\
\sum_{i=n}^{k} a_i &= -\sum_{i=k+1}^{n-1}a_i, \ \text{ if } k < n-1.
\end{split}
\end{align*}
By introducing the function 
\begin{align}\label{def:m}
m(t) =
\begin{cases} j, \ t_j \leq t < t_{j+1},\\
0, \ t \leq 0,
\end{cases}
\end{align}
we can write (\ref{def:PkQk}) as
\begin{align}\label{identity:PkQk}
\begin{split}
P_k(t) &= p_k + \sum_{i=k}^{m(t_k + t)-1} (p_{i+1} - p_i), \\
Q_k(t) &= q_k + \sum_{i=k}^{m(t_k + t)-1} (q_{i+1} - q_i).
\end{split}
\end{align}
Using the fact that $p_k = P_k(0)$ and $q_k = Q_k(0)$, along with the update (\ref{eq:stochastic_Hamiltonian_descent}), we can rewrite (\ref{identity:PkQk}) as
\begin{align}\label{identity:PkQk_mk}
\begin{split}
P_k(t) &= P_k(0) - \sum_{i=k}^{m(t_k + t)-1} \alpha_i \nabla F(q_i) + M_k(t) - \gamma  \sum_{i=k}^{m(t_k + t)-1} \alpha_i \nabla \varphi(p_i), \\
Q_k(t) &= Q_k(0) + \sum_{i=k}^{m(t_k + t)-1} \alpha_i \nabla \varphi(p_{i+1}),
\end{split}
\end{align}
where
\begin{align*}
M_k(t) = \sum_{i=k}^{m(t_k + t)-1} \alpha_i \delta M_i
\end{align*}
and  $\delta M_i = \nabla f(q_i,\xi_i) - \nabla F(q_i)$.

In the next section, we show that the process $\{M_k\}_{k\geq0}$ converges uniformly on compact sets, almost surely, to $0$.

\subsection{Convergence of the sequence $\{M_k\}$}\label{sec:Mk_converges}
The following lemma is an adaptation of part 1 of the proof of Theorem 2.1 from \citet{kushner_yin_2003}:
\begin{lemma}[Convergence of $\{M_k(t)\}_{k\geq 0}$]\label{lemma:convergence_of_Mk}
Suppose that Assumption \ref{ass:basic}, \ref{ass:kinetic_energy_function} and \ref{ass:step_size} holds, along with either Assumption \ref{ass:setting_one}, \ref{ass:setting_two} or \ref{ass:setting_three}.
Then, the sequence $\{M_k(t)\}_{k\geq 0}$ converges uniformly on compact sets almost surely to $0$.
More precisely, for any $T$ it holds that
\begin{align}\label{eq:convergence_of_Mk}
\lim_{k \to \infty} \sup_{t \in [0,T]} \lVert M_k(t) \rVert_2 = 0,
\end{align}
almost surely.
\end{lemma}
\begin{proof}[Proof of Lemma \ref{lemma:convergence_of_Mk}]
Closely following the proof of Theorem 2.1 in \citet{kushner_yin_2003}:
We let $\mathcal{F}_j = \sigma(\xi_1, \dots, \xi_j)$.
By definition, we have that
\begin{align*}
M_k(t)= \sum_{i=k}^{m(t_k+t)-1} \alpha_i \delta M_i,
\end{align*}
where $\delta M_i = \nabla f(q_i,\xi_i) - \nabla F(q_i)$.
Define 
\begin{align*}
\tilde{M}_j = \sum_{i=k}^{j} \alpha_i \delta M_i.
\end{align*}
We will show that $\tilde{M}_j$ is a martingale sequence. We first note that 
\begin{align*}
\mathbb{E} \left[ \tilde{M}_{j+1} | \mathcal{F}_j \right] = \tilde{M}_j,
\end{align*}
by Assumption \ref{ass:basic}.\ref{ass:unbiased} and the fact that the noise is independent. Next, we demonstrate that
\begin{align}\label{eq:Mjtildeismartingale}
\mathbb{E} \left[ \lVert \tilde{M}_{j+1} \rVert_2 \right] < \infty,
\end{align}

Note that
\begin{align*}
\mathbb{E} \left[
\lVert \tilde{M}_l \rVert_2^2 
\right]
&= 
\mathbb{E} \left[
\bigg\lVert\sum_{i=k}^{l} \alpha_i \delta M_i \bigg\rVert_2^2
\right]\\
&= 
\mathbb{E} \left[
\sum_{i=k}^{l} \alpha_i^2 \lVert \delta M_i \rVert_2^2
+
2 \sum_{i=k}^{l} \sum_{j=k}^{i-1} \alpha_i \alpha_j \langle \delta M_i 
, \delta M_j \rangle
\right] \\
&= 
\mathbb{E} \left[
\sum_{i=k}^{l} \alpha_i^2 \lVert \delta M_i \rVert_2^2
\right],
\end{align*}
where we have used the fact that for $j < i$ 
\begin{align*}
\mathbb{E}\left[\langle \delta M_i
, \delta M_j \rangle \right]
=
\mathbb{E}\left[ \mathbb{E}\left[ \langle \delta M_i
, \delta M_j \rangle | \mathcal{F}_j \right]\right]
=
\mathbb{E}\left[ \langle  \mathbb{E}\left[ \delta M_i  | \mathcal{F}_j \right]
, \delta M_j \rangle\right]
=0,
\end{align*}
since $M_j$ is $\mathcal{F}_j$-measurable and $ \mathbb{E}\left[ \delta M_i  | \mathcal{F}_j \right] = 0$ (recall that $\xi_i$ is independent of $\mathcal{F}_j$).
In the case that Assumption \ref{ass:setting_two}.\ref{ass:variance} holds we therefore have that
\begin{align}\label{eq:martingaledifferencebound}
\mathbb{E} \left[\lVert \delta M_i \rVert_2^2 
\right] < \infty.
\end{align}
Under Assumption \ref{ass:setting_three} we get that 
\eqref{eq:martingaledifferencebound} holds by Lemma \ref{lemma:2ndmomentsetting3}.
If instead Assumption \ref{ass:setting_one}.\ref{ass:expectedsmoothness} holds, we have that
\begin{align*}
\mathbb{E} \left[\lVert \delta M_{i+1} |\mathcal{F}_i \rVert_2^2 
\right] \leq \kappa(F(q_k) - F_*) + (1+\tau) \lVert \nabla F(q_k)\rVert_2^2 + \sigma^2.
\end{align*}
Under Assumption \ref{ass:setting_one}.\ref{ass:Lipschitzcontinuous_gradient} or \ref{ass:setting_two}.\ref{ass:L0L1smoothness} we get from Theorem \ref{thm:pqbounded} that the expectation of the right-hand side is finite\footnote{Under assumption \ref{ass:setting_one}.\ref{ass:Lipschitzcontinuous_gradient} we can use Lemma \ref{lemma:smoothbound} to bound the gradient with $2L(F(q_k) - F_*)$ which is bounded in expectation by Theorem \ref{thm:pqbounded}. 
}, in which case (\ref{eq:martingaledifferencebound}) also holds. Hence $\tilde{M}_j$ satisfies (\ref{eq:Mjtildeismartingale}) and it is thus a martingale.

We now show that (\ref{eq:convergence_of_Mk}) holds.
For any interval $[0,T]$, we have that
\begin{align*}
\sup_{t \in [0,T]} \lVert M_k(t)\rVert_2 = \sup_{k \leq j \leq l}\lVert \tilde{M}_j \rVert_2,
\end{align*}
where $l = m(t_k+ T)$. By \emph{Doob's submartingale inequality} \citep{kushner_yin_2003, williams_probability}, we have for every $\mu >0$ that
\begin{align*}
\mathbb{P}\left( \sup_{k \leq j \leq l}\lVert \tilde{M}_j \rVert_2 \geq \mu \right) 
\leq \frac{\mathbb{E}\left[ \lVert \tilde{M}_l \rVert_2^2 \right]}{\mu^2},
\end{align*}
which implies that
\begin{align*}
\mathbb{P}\left( \sup_{k \leq j }\lVert \tilde{M}_j \rVert_2 \geq \mu \right) 
\leq
 C\sum_{i=k}^{\infty} \alpha_i^2 
\end{align*}
and hence 
\begin{align*}
\lim_{k \to \infty}
\mathbb{P}\left( \sup_{k \leq j }\lVert \tilde{M}_j \rVert_2 \geq \mu \right) 
=0.
\end{align*}
By Theorem~1 in Section 2.10.3 of \citet{shiryaev_2016}, the sequence $\tilde{M}_j$ converges almost surely to $0$, 
i.e.\ there is a set $U$ such that $\mathbb{P}(U)=0$ and for every $\omega \in U^c$ 
we have that (\ref{eq:convergence_of_Mk}) holds.
\end{proof}


\subsection{Equicontinuity of the sequences $\{P_k\}_{k \geq 0}$ and $\{Q_k\}_{k \geq 0}$}\label{sec:PkQk_is_equicontinuous}
\lemmaPkQkequicontinuous*
To show Lemma \ref{lemma:PkQk_equicontinuous}, we make use of an equivalent definition of extended equicontinuity:
\begin{lemma}[Equivalent definition of extended continuity]\label{lemma:equivdefequicont}
A sequence of functions $\{f_k\}_{k \geq 0}$, $f_k: \R \to \R^d$  is  \emph{equicontinuous in the extended sense}
if and only if $\{|f_k(0)|\}_{k \geq 0}$ is bounded and for every $T$ and $\epsilon >0$ there is a null sequence $\left(a_k \right)_{k\geq 0}$ (that is, $\lim_{k \to \infty} a_k = 0$) such that
\begin{align}\label{nullseq}
 \sup_{0<|t-s| \leq \delta, \ t,s \in [0,T]} |f_k(t) - f_k(s)| \leq \epsilon + a_k.
\end{align}
\end{lemma}
\begin{proof}[Proof of Lemma \ref{lemma:equivdefequicont}]
By definition, (\ref{limsup}) is equal to
\begin{align*}
\lim_{k\to \infty} b_k \leq \epsilon
\end{align*}
with 
\begin{align*}
b_k := \sup_{j \geq k} \sup_{0<|t-s| \leq \delta, \ t,s \in [0,T]} |f_j(t) - f_j(s)|.
\end{align*}
Define
\begin{align*}
a_k = \max \{0,b_k - \epsilon \}.
\end{align*}
Then $\left(a_k \right)_{k\geq 0}$ satisfies all the requirements; $a_k$ is clearly positive and by  continuity of the function $\max \{0,x\}$ it holds that
\begin{align*}
 \lim_{k \to \infty} a_k  
 =
 \max \{0,  \lim_{k \to \infty}  b_k - \epsilon \} 
=0,
\end{align*}
as $\lim_{k \to \infty}  b_k - \epsilon \leq 0$.
Furthermore, we have that
\begin{align*}
 \sup_{0<|t-s| \leq \delta, \ t,s \in [0,T]} |f_k(t) - f_k(s)| \leq b_k \leq \epsilon + a_k,
\end{align*}
for every $k$ and thus we have shown that (\ref{nullseq}) follows from (\ref{limsup}).
We now show the converse.
Suppose that (\ref{nullseq}) holds. Taking the supremum of (\ref{nullseq}) we obtain
\begin{align*}
\sup_{j\geq k} \sup_{0<|t-s| \leq \delta, t,s \in [0,T]} |f_j(t) - f_j(s)| \leq \epsilon + \sup_{j\geq k}  a_j.
\end{align*}
We finally take the limit with respect to $k$
\begin{align*}
\lim_{k \to \infty} \sup_{j\geq k} \sup_{0<|t-s| \leq \delta, \ t,s \in [0,T]} |f_j(t) - f_j(s)| \leq \epsilon +  \lim_{k \to \infty} \sup_{j\geq k}  a_j
=: \epsilon +  \limsup_{k \to \infty}  a_k.
\end{align*}
But as $\lim_{k \to \infty}a_k$ exists by assumption and equals $0$, $ \limsup_{k \to \infty}  a_k=\lim_{k\to \infty}a_k=0$. We thus conclude that (\ref{limsup}) holds.
\end{proof}
We now turn to the proof of Lemma \ref{lemma:PkQk_equicontinuous}.
\begin{proof}[Proof of Lemma \ref{lemma:PkQk_equicontinuous}]
Closely following Lemma~2 in \citet{freise}:
We want to show that the sequence $\{Z_k\}_{k \geq 0}=\{(P_k,Q_k)\}_{k \geq 0}$, where $\{P_k\}_{k \geq 0}$ and $\{Q_k\}_{k \geq 0}$ are defined by (\ref{def:PkQk}), is equicontinuous in the extended sense.

First, we note that the sequences $\{P_k(0)\}$ and $\{Q_k(0)\}$ are finite except on a set of measure $0$, since by Theorem~\ref{thm:pqbounded} $\sup_k \lVert p_k \rVert_2 < \infty$ and $\sup_k \lVert q_k \rVert_2 < \infty$ almost surely. 

By Lemma \ref{lemma:equivdefequicont} an equivalent characterization of extended equicontinuity is that for every $\epsilon>0$, there is a sequence $\{a_k \}_{k \geq 0}$ such that $\lim_{k \to \infty} a_k = 0$ and a $\delta >0$ such that
\begin{align}\label{eq:Zk_equicont}
\sup_{|t-s|< \delta, \ t,s \in [0,T]}\lVert Z_k(t) - Z_k(s) \rVert_{\ell^2\left(\R^{2d}\right)} \leq \epsilon + a_k, \ \text{a.s.}
\end{align}
By (\ref{identity:PkQk_mk}), we have that
\begin{align}\label{eq:Pk_C_bound}
\lVert P_k(t) -P_k(s) \rVert_2 \leq C(\omega) \sum_{i=m(t_k+s)}^{m(t_k+t)-1}  \alpha_i  + \lVert M_k(t)\rVert_2 + \lVert M_k(s)\rVert_2,
\end{align}
where $C(\omega) = \sup_{i} \lVert \nabla F(q_i ) - \gamma \nabla \varphi(p_i) \rVert_2$.
By the boundedness of $p_k$ and $q_k$ along with the continuity of $\nabla F$ and $\nabla \varphi$, we have that $C(\omega) < \infty$, a.s.
The sum on the right-hand side of (\ref{eq:Pk_C_bound}) can be rewritten as
\begin{align*}
\sum_{i=m(t_k+s)}^{m(t_k+t)-1} \alpha_i
= 
t_{m(t_k+t)}- t_{m(t_k+s)}.
\end{align*}
By definition of $m$, (\ref{def:m}), we have that
\begin{align}\label{ineq:ts1}
t_{m(t_k+t)}- t_{m(t_k+s)}
\leq t_k +t -t_{m(t_k+s)}.
\end{align}
But we also have 
\begin{align*}
t_{m(t_k+s)} \leq t_k +s < t_{m(t_k+s)+1},
\end{align*}
and hence the right-hand side of (\ref{ineq:ts1}) can be rewritten and
bounded as follows
\begin{align*}
 t_k +t -(t_k+s) + (t_k+s) - t_{m(t_k+s)} 
 \leq (t-s) + t_{m(t_k+s)+1} - t_{m(t_k+s)}.
\end{align*}
Now, $t_{m(t_k+s)+1} - t_{m(t_k+s)} = \alpha_{m(t_k+s)+1}$ and hence we see that
\begin{align*}
\lVert P_k(t) -P_k(s) \rVert_2 \leq  C(\omega)  \left(|t-s| + \alpha_{m(t_k+s)+1} \right) + \lVert M_k(t) \rVert_2 + \lVert M_k(s) \rVert_2  .
\end{align*}
Let $\epsilon$ be greater than $0$. There are now two cases. If $C(\omega)=0$, (\ref{nullseq}) clearly holds for any $\delta>0$. If $C(\omega) \neq 0$, then take $\delta>0$ so small that $C(\omega) \delta < \epsilon$. We then have
\begin{align*}
  &\sup_{|t-s|<\delta, \ t,s \in [0,T]} \lVert P_k(t) -P_k(s) \rVert_2 \\
  &\quad\leq  \sup_{|t-s|<\delta, t,s \in [0,T]} \left( C(\omega)(|t-s| + \alpha_{m(t_k+s)+1}) + \lVert M_k(t) \rVert_2 + \lVert M_k(s) \rVert_2 \right)
\\
&\quad< \epsilon
+
C(\omega) \alpha_{m(t_k)+1} + 2 \lVert M_k(T) \rVert_2.
\end{align*}
By Lemma \ref{lemma:convergence_of_Mk}, $\lim_{k \to \infty}\lVert M_k(T)\rVert_2=0$ a.s.\ and
we see that (\ref{nullseq}) in Lemma \ref{lemma:equivdefequicont} holds almost surely.
A similar argument yields an analogous bound for $Q_k$, and by the equivalence of norms on $\R^{2d}$, we obtain (\ref{eq:Zk_equicont}).
\end{proof}

In the next section, we show that the processes $\{P_k\}_{k\geq 0}$ and $\{Q_k\}_{k\geq 0}$ can be written as solutions to the integral equations  corresponding to (\ref{eq:nearly_Hamiltonian_system}), plus terms that converge uniformly on compact sets to $0$ as $k \to \infty$.
\subsection{Asymptotic solution}\label{appendix:asymptotic_solution}
\lemmaPkQkasymptoticsolutions*
\begin{proof}[Proof of Lemma \ref{lemma:PkQk_asymptotic_solutions}]
We start by showing that the sum 
\begin{align*}
-\sum_{i=k}^{m(t_k+t)-1}  \alpha_i \nabla F(q_i )
\end{align*}
in Equation (\ref{identity:PkQk_mk})
can be rewritten as 
\begin{align*}
-\int_0^t \nabla F (Q_{k}(s))\mathrm{d}s + \mu_{1,k}(t),
\end{align*}
where $Q_k(t)$ is defined by (\ref{def:PkQk}) and $\{\mu_{1,k}\}_{k \geq 0}$ is a sequence of functions that tends to 0 uniformly on compact intervals.
Consider
\begin{align*}
I_k := -\int_0^t \nabla F (Q_{k}(s)) \mathrm{d}s .
\end{align*}
Then, since $t_k + s$ belongs to a single interval $[t_i,t_{i+1})$,
\begin{align*}
I_k &= -\int_0^t \left( \nabla F (q_0) I_{(-\infty,t_0)}(t_k+s) - \sum_{i=0}^{\infty} \nabla F (q_i) I_{[t_i,t_{i+1})}(t_k +s) \right) \mathrm{d}s \\
&= -\int_0^t \left( \nabla F(q_0) I_{(-\infty,t_0-t_k)}(s) - \sum_{i=0}^{\infty} \nabla F(q_i) I_{[t_i-t_k,t_{i+1}-t_k)}(s) \right) \mathrm{d}s.
\end{align*}
The term $t_0-t_k$ is always less than or equal to $0$. Hence the first term disappears as we are integrating from $0$ to $t$. 
For $i<k$, we have that $t_{i+1}-t_k \leq 0$. We can therefore start the sum at $i=k$, as earlier terms will not contribute to the integral.
Thus, 
\begin{align*}
I_k = - \int_0^t \left( \sum_{i=k}^{\infty} \nabla F(q_i) I_{[t_i-t_k,t_{i+1}-t_k)}(s) \right) \mathrm{d}s.
\end{align*}
Now suppose that $t_j - t_k\leq t < t_{j+1}-t_k$. We split up the previous integral as follows:
\begin{align*}
  I_k &= -\int_0^{t_j-t_k} \left( \sum_{i=k}^{j-1} \nabla F(q_i) I_{[t_i-t_k,t_{i+1}-t_k)}(s) \right) \mathrm{d}s \\
  &\quad- \int_{t_j-t_k}^t \nabla F(q_k) I_{[t_j-t_k,t_{j+1}-t_k)}(s)  \mathrm{d}s
\\
&=-\sum_{i=k}^{j-1} \nabla F(q_i) \alpha_i
-
 \nabla F(q_j)\left( t- t_j + t_k \right),
\end{align*}
where we have used that $\int_0^{t_j-t_k}  I_{[t_i-t_k,t_{i+1}-t_k)}(s) ds = \alpha_i$.
Using the fact that $m(t+ t_k) = j$, we can rewrite this further as
\begin{align*}
I_k = -\sum_{i=k}^{m(t_k+t)-1} \nabla F(q_i) \alpha_i
-
\mu_{1,k}(t),
\end{align*}
where $\mu_{1,k}(t) =  \nabla F (q_{m(t_k+t)})\left( t- t_{m(t_k+t)} + t_k \right)$.
The function $\mu_{1,k}$ is piecewise linear and $0$ at $t=t_j - t_k$. 
Since $\sup_k \lVert q_k \rVert_2 < \infty$ a.s. by Theorem~\ref{thm:pqbounded}, the set $\{q_k\}$ can be enclosed in a compact set. As $\nabla F$ is continuous in each of the settings that we consider, a positive random variable $C(\omega)$ which satisfies
\begin{align*}
  \lVert\nabla F(q_{m(t_k+t)})\rVert_2 \leq C(\omega) < \infty, \ \text{a.s.},
\end{align*}
thus exists.
Hence, it holds that
\begin{align*}
 \lVert \mu_{1,k}(t) \rVert_2
\leq C(\omega) | \alpha_{m(t_k+t)} |.
\end{align*}
Now, we have that $\lim_{k \to \infty }\sup_{t \in [0,T]} \alpha_{m(t_k+t)} = 0$ for any fixed $T$, since $\lim_{k \to \infty} \alpha_k = 0$, and thus $\mu_{1,k}$ converges to $0$ uniformly on  compact intervals.

In a similar fashion we obtain for the last term in (\ref{identity:PkQk_mk}) that
\begin{align*}
-\sum_{i=k}^{m(t_k+t)-1}  \alpha_i \nabla \varphi(p_i )
=
-\int_0^t \nabla \varphi (P_k(s))\mathrm{d}s + \mu_{2,k}(t),
\end{align*}
where $\{\mu_{2,k}\}_{k\geq 0}$ converges uniformly on compact sets to $0$.
Letting $\mu_k = \mu_{1,k} + \gamma \mu_{2,k}$, we obtain the expression in the first line of (\ref{eq:PkQk_asymptotic_solutions}).

We now turn our attention to the second line of (\ref{eq:PkQk_asymptotic_solutions}). By an argument analogous to the previous, we can write the second line of (\ref{identity:PkQk_mk}) as 
\begin{align*}
Q_k(t) = Q_k(0) + \int_0^t \varphi(P_{k+1}(s))\mathrm{d}s + \nu_k(t),
\end{align*}
where $\nu_k$ converges uniformly on compact sets to $0$.
We can rewrite the integral on the right-hand side as
\begin{align*}
\int_0^ t \nabla \varphi(P_{k+1}(s)) \mathrm{d}s 
= 
\underset{:=\kappa_k(t)}{\underbrace{\int_0^ t \nabla \varphi(P_{k+1}(s)) -
 \nabla \varphi(P_k(s)) \mathrm{d}s}}
+
\int_0^ t \nabla \varphi(P_k(s))\mathrm{d}s.
\end{align*}
By the Lipschitz continuity of $\nabla \varphi$ and the definition of $P_k$ (\ref{def:PkQk}),
\begin{align*}
\lVert \kappa_k(t) \rVert_2
\leq
\int_0^ t \lambda \lVert P_{k+1}(s) -
 P_k(s) \rVert_2 \mathrm{d}s
 =
 \int_0^ t \lambda \lVert P_k(\alpha_{k} + s) -
 P_k(s) \rVert_2 \mathrm{d}s,
\end{align*}
Since $\{P_k(t)\}_{k\geq 0}$ is equicontinuous in the extended sense by Lemma \ref{lemma:PkQk_equicontinuous}, there is for each $T$ and $\epsilon>0$ a $\delta>0$ such that
 \begin{align}\label{eq:Pkequicont}
\limsup_{k \to \infty} \sup_{|t-s|< \delta, \ t,s \in [0,T]} \lVert P_k(\alpha_{k} + s) -
 P_k(s) \rVert_2 \leq \epsilon.
 \end{align}
What remains is to show is that $\{\kappa_k\}_{k \geq 0}$ converges uniformly on compact sets to $0$.
For any $T$, we have that
\begin{align*}
\lim_{k \to \infty} \sup_{t \in [0,T]} \lVert \kappa_k(t)\rVert_2
\leq 
\lim_{k \to \infty} 
 \int_0^{T} \lambda \lVert P_k(\alpha_{k+1} + s) -
 P_k(s) \rVert_2 \mathrm{d}s
\end{align*}
since the integrand is positive. By Theorem~\ref{thm:pqbounded}, we can bound $\lVert P_k(\alpha_{k+1} + s) -
 P_k(s) \rVert_2 \leq 2 \sup_{t \in \mathbb{R}} \lVert 
 P_k(t) \rVert_2  < \infty$. Thus, we can use the Lebesgue dominated convergence theorem and take the limit inside the integral:
\begin{align*}
\lim_{k \to \infty} \sup_{t \in [0,T]} \lVert \kappa_k(t) \rVert_2
\leq 
 \int_0^ T \lim_{k \to \infty}  \lambda \lVert P_k(\alpha_{k+1} + s) -
 P_k(s) \rVert_2 \mathrm{d}s
\end{align*}
By (\ref{eq:Pkequicont}), we can make the integrand arbirarily small by choosing $k$ so large that $\alpha_{k+1} \leq \delta$.
\end{proof}

\subsection{Convergence to a locally asymptotically stable set}

The goal of this section is to show
\thmkushneryin*

We start by showing the following help-lemma:
\begin{lemma}\label{lemma:enters_compact_set_infinitelyoften}
In the same context as Theorem~\ref{thm:kushner_yin}, for each $\delta>0$ there is a subsequence $\{z_{n_k}\}_{k \geq 0}$ of $\{z_k\}_{k \geq 0}$ such that that  $\{z_{n_k}\}_{k \geq 0} \subset N_{\delta}(A)$.
\end{lemma}
\begin{proof}[Proof of Lemma \ref{lemma:enters_compact_set_infinitelyoften}]
Since $\{z_{n_k}\}_{k \geq 0} \subset K$, and $K$ is compact, we can find a further subsequence  
$\{z_{n_k'}\}_{k \geq 0}$ that tends to $z_0 \in K$. Let $\{Z_{n_k'}\}_{k \geq 0}$ be the sequence of shifted interpolations associated with $\{z_{n_k'}\}_{k \geq 0}$. This family is equicontinuous in the extended sense by Lemma~\ref{lemma:PkQk_equicontinuous}, and thus it has a subsequence $\{Z_{n_k^{''}}\}_{k \geq 0}$ converging to a function $z(\cdot)$ which is a solution to (\ref{eq:nearly_Hamiltonian_system}) by Lemma~\ref{lemma:PkQk_asymptotic_solutions}, and satisfies $z(0)=z_0$. Since $z_0$ is in the domain of attraction of $A$, it holds that
\begin{align*}
\lim_{t \to \infty} \inf_{a \in A} \lVert z(t) - a\rVert_{\ell^2\left(\R^{2d}\right)} = 0.
\end{align*}
Choose $T_{\frac{\delta}{2}}$ so large that 
\begin{align*}
\inf_{a \in A} \lVert z(t) - a\rVert_{\ell^2\left(\R^{2d}\right)} < \frac{\delta}{2}
\end{align*}
for all $t \geq T_{\frac{\delta}{2}}$.
Then, we have
\begin{align*}
\inf_{a \in A} \lVert Z_{n_k''}(t) - a\rVert_{\ell^2\left(\R^{2d}\right)} 
&\leq
                                                                            \lVert Z_{n_k''}(t) - z(t)\rVert_{\ell^2\left(\R^{2d}\right)}
+ \inf_{a \in A} \lVert z(t) - a\rVert_{\ell^2\left(\R^{2d}\right)}\\
&< \lVert Z_{n_k''}(t) - z(t)\rVert_{\ell^2\left(\R^{2d}\right)} 
+ \frac{\delta}{2}.
\end{align*}
Since $\{Z_{n_k''}\}$ converges uniformly on compact sets to $z$, we can choose $N_{\frac{\delta}{2}}$ so large that
for any $n_k' \geq N_{\frac{\delta}{2}}$ we have that $\sup_{s \in [0,t]} \lVert Z_{n_k''}(s) - z(s)\rVert_{\ell^2\left(\R^{2d}\right)}  < \frac{\delta}{2}$.
Hence, for $n_k'' > N_{\frac{\delta}{2}}$ we have
\begin{align*}
\inf_{a \in A} \lVert Z_{n_k''}(t) - a\rVert_{\ell^2\left(\R^{2d}\right)} <
\delta,
\end{align*}
which yields the statement of the lemma.
\end{proof}

With the help of Lemma \ref{lemma:enters_compact_set_infinitelyoften}, we now show Theorem \ref{thm:kushner_yin}. The proof is inspired by the proof strategy in~\cite{FortPages.1996}.

\begin{proof}[Proof of Theorem \ref{thm:kushner_yin}]

Let $\epsilon>0$ and let $\delta >0$ be as in Definition \ref{def:locally_asymptotically_stable}.
According to Lemma \ref{lemma:enters_compact_set_infinitelyoften}, there is a subsequence $\{z_{r_k}\}_{k \geq 0}$ of $\{z_k\}_{k \geq 0}$ such that
$\{z_{r_k}\}_{k \geq 0} \subset N_{\delta / 2}(A)$.

We now show by contradiction that $\{z_k\}_{k \geq 0}$ cannot escape $N_{\epsilon}(A)$ infinitely often.
Suppose that there is a subsequence 
$\{z_{s_k}\}_{k \geq 0} \subset N_{\epsilon}(A)^c$.

Define $\ell_0 = \min \{j: z_j \in N_{\delta / 2}(A) \} $ and recursively for $k = 1, 2, \ldots$,
\begin{align*}
n_k &= \min \{j : j \geq \ell_{k-1} \text{ and }  z_j \in N_{\epsilon}(A)^c \} ,\\
m_k &= \max \{j : j\leq n_k \text{ and } z_j \in N_{\delta / 2}(A) \},\\
\ell_k &= \min \{j : j\geq n_k \text{ and } z_j \in N_{\delta / 2}(A) \}.
\end{align*}
Then there is no index $j \in \{m_k+1, \dots, n_k\}$ such that
$z_j \in N_{\delta / 2}(A)$; i.e.\ $m_k$ is the last index for which $z_j$ visits $N_{\delta / 2}(A)$ before going to $ N_{\epsilon}(A)^c$.

Consider the associated sequence of functions $\{Z_{m_k}\}_{k \geq 0}$. This satisfies
\begin{align*}
  Z_{m_k}(t) &= z_{m_k} \in N_{\delta / 2}(A) , & 0 &\leq t < \alpha_{m_k}, \\
  Z_{m_k}(t) &= z_{n_k} \in N_{\epsilon}(A)^c, & t_{n_k} - t_{m_k} &\le t < t_{n_k} - t_{m_k} + \alpha_{n_k}.
\end{align*}
In between these two time intervals, $Z_{m_k}$ attains the values $z_{m_k + 1}$, $z_{m_k + 2}$, $\ldots$, $z_{n_k-1}$. We can therefore guarantee that
\begin{align}\label{fact1}
Z_{m_k}(t) \in  N_{\delta / 2}(A)^c \text{ for } t \in [\alpha_{m_k},t_{n_k} - t_{m_k}].
\end{align}

Let $\{Z_{m_k'}\}$ be a subsequence of $\{Z_{m_k}\}$ that converges uniformly on compact sets to a function $z$.
First, assume that $\limsup_{k \to \infty} t_{n_k'} - t_{m_k'} = \infty$. Then we can extract a subsequence (which we continue to denote $\{t_{n_k}\}$) for which $\lim_{k \to \infty} t_{n_k'} - t_{m_k'} = \infty$. 
Under this assumption, it holds that 
\begin{align*}
z(t) \in N_{\delta / 2}(A)^c \text{ for }t >0.
\end{align*}
If this was not the case there would be a $t' >0$ such that $z(t') \in N_{\delta / 2}(A)$. 
By the openness of $N_{\delta / 2}(A)$ we can choose $\eta>0$ such that
$B(z(t'),\eta) \subset N_{\delta / 2}(A)$. There is a $K_{\eta}$ such that $k \geq K_{\eta}$ implies that
\begin{align*}
\lVert Z_{m_k'}(t') - z(t')\rVert < \eta,
\end{align*}
i.e.\ $ Z_{m_k'}(t') \in B(z(t'),\eta) \subset N_{\delta / 2}(A)$.
However, since $t_{n_k'} - t_{m_k'} \to \infty$ and $\alpha_{m_k'} \to 0$, it holds that $t' \in [\alpha_{m_k'},t_{n_k'} - t_{m_k'}]$ for large enough $k$. This contradicts (\ref{fact1}), so that indeed $z(t) \in N_{\delta / 2}(A)^c$ for $t >0$. By Theorem \ref{thm:discontinuos_arzela_acoli}, $z$ is continuous and since $Z_{m_k}(0) \in N_{\delta / 2}(A)$ we must thus have $z(0) \in \partial N_{\delta / 2}(A)$.
However, the fact that $z(t) \in N_{\delta/2}(A)^c$ for $t\geq 0$ contradicts the asymptotic stability of $A$, since this path which starts in $\partial N_{\delta / 2}(A) \subset N_{\delta}(A)$ does not approach $A$.
This is a contradiction towards our assumption that $t_{n_k'} - t_{m_k'} \to \infty$, and we can thus define $\tilde{T} = \sup_k t_{n_k'} - t_{m_k'} < \infty$.
Then $[0, \tilde{T}]$ is a compact interval such that $\{ t_{n_k'} - t_{m_k'}\}_{k\geq 0} \subset [0, \tilde{T}]$. Hence there is a subsequence 
 $\{ t_{n_k^{''}} - t_{m_k^{''}}\}_{k\geq 0} \subset  \{ t_{n_k'} - t_{m_k'}\}_{k\geq 0}$ that converges to some $T \in [0, \widetilde{T}]$.

The corresponding sequence of functions $\{ Z_{m_k^{''}}\} $ is a subsequence of $\{Z_{m_k'}\}$ and thus it must also converge uniformly on compact sets to the same function $z$. 
From the uniform convergence it also follows that
\begin{align*}
z_{n_k^{''}} = Z_{m_k^{''}}(t_{n_k^{''}}- t_{m_k^{''}}) \to z(T). 
\end{align*}
Since each $z_{n_k^{''}} \in N_{\epsilon}(A)^c$ we must have $z(T) \in N_{\epsilon}(A)^c$. 
However, this contradicts the Lyapunov stability of $A$ and therefore also our original assumption that there exists a subsequence $\{z_{s_k}\}_{k \geq 0} \subset N_{\epsilon}(A)^c$. This concludes the proof. 
\end{proof}

\subsection{Convergence to a stationary point}

We will now apply  Theorem \ref{thm:kushner_yin} and show that $\{z_k\}_{k \geq 0}$ converges to the set $\{z: H(z) \leq \liminf_k H(z_k)\}$. First, we need to show that it is locally asymptotically stable:
\lemmalocallyasymptoticallystable*
\begin{proof}
We need to show that for all $\epsilon > 0$ we can choose $\delta>0$ so that if $z_0 \in N_{\delta}(A)$ then $z(t)$ stays in $N_{\epsilon}(A)$ and that $\lim_{t \to \infty} z(t) \in A$.

By Lemma \ref{lem:deltanbhofsublevelsets2}, there exists a $\eta >0$ such that
$\{z: H(z) \leq c + \eta\} \subset N_{\epsilon}(A)$. 
Now $z(t)$ will stay in $\{z:H(z) \leq c + \eta\}$ since $z(t)$ decreases along the paths of $H$. 
However, it might not converge to $A$ since there may exist stationary points $z_*$ such that
\begin{align*}
c < H(z_*) \leq c+ \eta,
\end{align*}
and if we reach one of these we will get stuck there instead of reaching $A$. Define
\begin{align*}
c_* = \inf \{H(z_*): c < H(z_*) \leq c + \eta , \nabla H(z_*)=0\}.
\end{align*}
It holds that $c_* > c$, i.e.\ we cannot find stationary points for which $H(z_*)$ is arbitrarily close to $c_*$. We can see this by letting $\Lambda = \{x: \nabla H(x) = 0\}$ and $K= [c , c+ \eta]$. Then by Assumption \ref{ass:basic}.\ref{ass:critical_points_locallyfinite}, there exist numbers $\{y_i\}_{i=1}^n$ , such that $y_1 < \dots < y_n$ and
\begin{align*}
\{ H(z): c \leq H(z) \leq c+ \eta : z \in \Lambda \}
= H(\Lambda) \cap K = \{y_1, \dots, y_n \}.
\end{align*}
If $y_1 = c$, we have that
\begin{align*}
y_2 = \min H(\Lambda) \cap K = \min \{H(z): c < H(z) \leq c + \eta, \ z \in \Lambda\} = c_*
\end{align*}
and thus $c_* >c$. Similarly, if $y_1 > c$, we also get that $c_* = y_1 > c$. Thus indeed it always holds that $c_* > c$, and we can take $\mu >0$ so small that $c_*-\mu >c$.
By Lemma \ref{lem:deltanbhofsublevelsets1} there exists some $\delta >0$ such that
\begin{align*}
N_{\delta}(A) \subset \{H(z) < c_* - \mu\}
\end{align*}
Since $H$ is decreasing along the paths of $z(t)$, any solution starting in $N_{\delta}(A)$ will stay inside $\{z:H(z) \leq c_* - \mu\}$ (and thus $N_{\epsilon}(A)$). 
By La Salle's invariance principle, any path starting in the compact set $\mathcal{M} =\{z: H(z) \leq c_* - \mu\}$ tends to $\{z \in \mathcal{M}: \nabla H(z) =0\}$.
All points $z_* \in \{z \in \mathcal{M}: \nabla F(z) =0\}$ satisfy $H(z_*) \leq c$, by the choice of $c_*$ and $\mu >0$. 
Thus, $z(t) \to \{z: H(z) \leq c\}$ whenever $z(0) \in N_{\delta}(A)$.

We have for the given $\epsilon>0$ found a $\delta>0$ such that any path in $N_{\delta}(A)$ never leaves $N_{\epsilon}(A)$ and tends to $A$ as $t \to \infty$.
\end{proof}

We are now ready to prove our main result, Theorem~\ref{thm:main}:
\thmmain*

\begin{proof}
Let $c = \liminf_k H(z_k)$.
We start by showing that $\lim_{k \to \infty} H(z_k) = c$.
By Lemma \ref{lemma:locally_asymptotically_stable}, the set $A=\{z: H(z) \leq c\}$ is a locally asymptotically stable set, and by Lemma \ref{lemma:infcio}, we can find a compact set $K$ in the domain of attraction of $A$ that $\{z_k\}_{ k \geq 0}$ enters infinitely often:
In particular, we can take $K = \mathcal{M}$, where $\mathcal{M}$ is as in the proof of Lemma \ref{lemma:locally_asymptotically_stable}; $\mathcal{M}$ is in the domain of attraction of $A$, and by Lemma \ref{lemma:infcio}, $\{z_k\}_{ k \geq 0}$ visits $\mathcal{M}$ infinitely often.
Theorem \ref{thm:kushner_yin} then implies that 
 $z_k \to \{z: H(z) \leq c\}$. 
Suppose that $\lim_k H(z_k)\neq c$.
The negation of the statement is
\begin{align*}
\exists \epsilon >0 : \forall n \, \exists n_k\geq n, \bigl(H(z_{n_k}) \leq c - \epsilon\bigr) \lor 
\bigl(H(z_{n_k}) \geq c + \epsilon\bigr).
\end{align*}
This is false, since in the case that there exists a subsequence $\{z_{n_k}\}$ such that $H(z_{n_k}) \geq c+ \epsilon$,  $\{z_k\}$ would not converge to $\{z: H(z) \leq c\}$.
Further, if there exists a subsequence that satisfies $H(z_{n_k}) \leq c - \epsilon$ we would have
$\liminf_k H(z_k) \leq c - \epsilon < c$ which is also a contradiction by the choice of $c$.

We now know that $\lim_{k\to\infty}H(z_k) =c$, but we have yet to verify that $\{z_k\}_{k \geq0}$ converges to the set of stationary points.
Suppose that this is not the case. For brevity let $\Lambda = \{ z: \nabla H(x) = 0\}$. Then there exists an $\epsilon_0 >0$ and a subsequence $\{z_{n_k}\}_{k \geq 0}$ 
such that
\begin{align}\label{ineq:znknotstationary}
\inf_{x \in \Lambda} \lVert z_{n_k} - x \rVert \geq \epsilon_0.
\end{align}
From the previous paragraph, it holds that 
\begin{align*}
\lim_{n_k \to \infty} H(z_{n_k}) = c.
\end{align*}
By Theorem \ref{thm:pqbounded}, the sequence $\{z_{n_k}\}_{k\geq 0}$ is bounded. By Lemma \ref{lem1} we can thus find a further subsequence (still denoted by $\{z_{n_k}\}_{k \geq 0}$) and a point
$ \tilde{z}_0$ such that $\lim_{n_k \to \infty} z_{n_k} =\tilde{z}_0$ and $H(\tilde{z}_0) = c$.
The sequence of interpolations $\{Z_{n_k}(\cdot)\}$ 
associated with $\{ z_{n_k} \}$, has a subsequence $\{Z_{n_k'}(\cdot)\}$ that converges to a solution $\tilde{z}(\cdot)$, such that  $\tilde{z}(0)= \tilde{z}_0$.
By (\ref{ineq:znknotstationary}) it holds that $\tilde{z}_0 \not\in \{z:\nabla H(z)=0\}$.
As $H$ is decreasing along the paths of $\tilde{z}(\cdot)$, we have for $t' >0$ that $c = H(\tilde{z}_0) = H(\tilde{z}(0))> H(\tilde{z}(t'))$.
However, $\tilde{z}(\cdot)$ is taking values in $L\left( \{z_k\}\right)$, the set of limit points of $\{z_k\}$, compare Proposition 1.b) in \cite{FortPages.1996}. Thus, there is some subsequence $\{z_{m_k}\}$ that converges to 
$H(\tilde{z}(t'))$.
But since $c = H(\tilde{z}_0) > H(\tilde{z}(t'))$ and $\{z_{m_k}\}$ converges to 
$H(\tilde{z}(t'))$. This is a contradiction, by the choice of $c$.
It follows that the set of limit points $L\left(\{z_k\}_{k \geq 0} \right)$ of $\{z_k\}_{k \geq 0}$ is contained in $\{z: \nabla H(z) =0\}$.
Since $z_{k+1} -z_k \to 0$, the limit set  $L\left(\{z_k\}_{k \geq 0} \right)$ is connected \citep{asic_adamovic_1970}. By Assumption \ref{ass:isolated_equilibria}, this implies that $\{z_k\}_{k \geq 0}$ converges to a single stationary point.
\end{proof}

At last, we prove Corollary \ref{cor:convinexp}:
\corconvinexp*                   
\begin{proof}
By Theorem~\ref{thm:main}  $\{q_k\}_{k \geq 0}$ converges almost surely to a stationary point $q_*$. Hence,
\begin{align*}
\lim_{k \to \infty} \lVert \nabla F(q_k) \rVert_2 =0, \text{a.s.},
\end{align*}
compare Lemma~2.3 in \cite{van2000asymptotic}.
Under Assumption~\ref{ass:setting_one}, we have from Lemma \ref{lemma:smoothbound} that
\begin{align*}
\lVert \nabla F(q_k)\rVert_2^2 \leq 2L  \left( F(q_k) - F_* \right),
\end{align*}
and by Theorem~\ref{thm:pqbounded}, 
\begin{align*}
\sup_k \E \left[ \lVert \nabla F(q_k)\rVert_2^2 \right] < \infty.
\end{align*}
By Lemma 3 in Chapter 2.6 of \cite{shiryaev_2016} with $G(t) = t^2$ we obtain that the sequence $\{\| \nabla F(q_k) \|\}_{k \geq 0}$ is uniformly integrable. It follows from Theorem 5 in Chapter 2.6 of \cite{shiryaev_2016} that
\begin{align*}
\lim_{k \to \infty} \E \left[ \lVert \nabla F(q_k) \rVert_2 \right] =0.
\end{align*}
Under Assumption~\ref{ass:setting_two} or \ref{ass:setting_three}, we instead get from Lemma \ref{lemma:smoothbound} and Theorem \ref{thm:pqbounded} that
\begin{align*}
\sup_k \E \left [\lVert \nabla F(q_k)\rVert_2\right] < \infty.
\end{align*}
It follows that
\begin{align*}
\lim_{k \to \infty}  \E \left [ \lVert \nabla F(q_k)\rVert_2^{\frac{1}{2}} \right]= 0.
\end{align*}
\end{proof}

\section{Auxiliary results}\label{appendix:auxiliary}
Several of the results in this section are relatively standard, but we include them for the convenience of the reader.
\begin{lemma}\label{lem1}
Let $\{x_k\}_{k\geq 0}$ be a sequence in  $\mathbb{R}^d$.
Suppose that $\sup_k \lVert x_k \rVert  <\infty$ and that $f(x_k) \to y$, where $f:\mathbb{R}^d \to \mathbb{R}$ is continuous. Then there is a subsequence $\{x_{n_k}\}_{k\geq 0} \subset \{x_k\}_{k\geq 0}$ that converges to a number $x$ such that
$f(x) = y$
\end{lemma}
\begin{proof}
Since  $\sup_k \lVert x_k \rVert  <\infty$ there is a compact set $K$ such that $ \{x_k\} \subset K$. By compactness, there is a subsequence $\{x_{n_k}\}$, that converges to an element $x$. The sequence  $\{f(x_{n_k})\}$ is a subsequence of $ \{f(x_k)\} $, and must converge to the same limit $y$. However, by continuity of $f$, we have that $\lim_{k \to \infty} f(x_{n_k}) = f ( \lim_{k \to \infty}  x_{n_k}) = f(x)$. Thus, $f(x) = y$.
\end{proof}
The next two Lemmas, Lemma \ref{lem:deltanbhofsublevelsets1} and \ref{lem:deltanbhofsublevelsets1}, are helpful in showing that the sublevel sets the Hamiltonian are locally asymptotically stable:
\begin{lemma}\label{lem:deltanbhofsublevelsets1}
Suppose that $f:\mathbb{R}^d \to \mathbb{R}$ is continuous and coercive. Let $A = \{x: f(x) \leq c\}$, where $c$ is such that $A \neq \emptyset$. Then, for every $\eta >0$ there is $\delta>0$ such that $N_{\delta}(A) \subset \{x: f < c+ \eta\}$.
\end{lemma}
\begin{proof}
We first note that since $f$ is coercive, $A$ is compact. Let $\eta>0$ be given. By continuity of $f$, there is $\delta>0$ such that 
\begin{align*}
|x-y|< \delta \ \implies |f(x) - f(y)| < \eta.
\end{align*}
For such $\delta$, we consider
\begin{align*}
N_{\delta}(A) = \{x: \inf_{a \in A} \lVert x - a \rVert < \delta\}.
\end{align*}
Take $x_0 \in N_{\delta}(A)$. Then $ \inf_{a \in A} \lVert x_0 - a \rVert < \delta$ and by the definition of the infimum, there exists for each $n$ and element $a_n \in A$ such that
\begin{align*}
\lVert x_0 - a_n \rVert  <  \inf_{a \in A} \lVert x_0 - a \rVert  + \frac{1}{n}.
\end{align*}
Then $\{a_n\}\subset A$ and by compactness there is a subsequence $\{a_{n_k}\}$ that converges to an element $a_* \in A$. Since
\begin{align*}
\lVert x_0 - a_{n_k} \rVert  <  \inf_{a \in A} \lVert x_0 - a \rVert  + \frac{1}{n_k},
\end{align*}
it holds that
\begin{align*}
\lVert x_0 - a_* \rVert  \leq  \inf_{a \in A} \lVert x_0 - a \rVert  < \delta.
\end{align*}
Since $f$ is continuous, we have that 
\begin{align*}
f(x_0) < f(a_*) + \eta \leq c + \eta
\end{align*}
i.e. $x_0 \in \{x: f(x) \leq c+ \eta\}$. Thus $N_{\delta}(A) \subset \{x: f(x) \leq c+ \eta\}$.
\end{proof}

\begin{lemma}\label{lem:deltanbhofsublevelsets2}
Let $A = \{x:f(x) \leq c\}$ (where $c$ is such that $A \neq \emptyset$). Then for every $\epsilon >0$ there is $\eta>0$ such that $\{x: f(x) \leq c + \eta\} \subset N_{\epsilon}(A)$.
\end{lemma}

\begin{proof}
If this was not the case, then there exists some $\epsilon>0$ and for every $n$ we can find $x_n$ that satisfies
 \begin{align*}
 x_n \in \biggl\{x: f(x)\leq c +\frac{1}{n} \biggr\} \cap N_{\epsilon}(A)^c
 \subset \{x: f(x)\leq c +1 \} \cap N_{\epsilon}(A)^c.
 \end{align*}
 The latter is compact since $ N_{\epsilon}(A)^c$ is closed and $\{x:f(x)\leq c +1 \}$ is compact.
 Thus, $\{x_n\}$ has a subsequence $\{x_{n_k}\}$ that converges to $x_* \in \{x: f(x)\leq c +1 \} \cap N_{\epsilon}(A)^c \subset  N_{\epsilon}(A)^c  \subset A^c$.
 However, each $x_{n_k} \in \{x: f(x)\leq c +\frac{1}{n_k} \}$ and thus by continuity it holds that
 $f(x_*) \leq c$ which is a contradiction.
\end{proof}
We will use the next lemma to show that under the given assumptions, the sublevel set $\{z: H(z) \leq \liminf_{k \to \infty} H(z_k)\}$ is non-empty:
\begin{lemma}\label{lem:liminfnonempty}
Let $f:\mathbb{R}^d \to \mathbb{R}$ be a continuous function that is bounded below by $f_* = \inf_{x \in \mathbb{R}^d} f(x)$.
Let  $\{x_k\}_{k \geq 0}$ be a sequence in $\mathbb{R}^d$ with $\sup_{k} \lVert x_k \rVert < \infty$. Put $c = \liminf_{k} f(x_k) $. Then $\{x: f(x) \leq c\} \neq \emptyset$.
\end{lemma}
\begin{proof}
By assumption  $\{x_k\}_{k \geq 0}$ is contained in a compact set $K$. By continuity, it holds that $C=\sup_k f(x_k) < \infty$ and 
hence the sequence $\{f(x_k)\}_{k \geq 0}$ is contained in the (compact) interval $[f_*,C]$.
It follows that $c \in [f_*,C]$. By a standard result in real analysis, we can (since  $\{f(x_k)\}_{k \geq 0}$ is bounded) find a subsequence $\{f(x_{n_k})\}$ that converges to $c$. By Lemma \ref{lem1}, there exists a further subsequence $\{x_{n_k'}\}$ that converges to an element $x$ such that $f(x) = c$.
Hence $x \in \{x: f(x) \leq c\}$ and $\{x: f(x) \leq c\} \neq \emptyset$.
\end{proof}
\begin{lemma}\label{lemma:infcio}
Let $f$ be a function which is bounded below. and let $c = \lim \inf_k f(x_k)$. Then for every $\delta >0$, the sequence $\{x_k\}$ is in the set $A_{\delta}=\{x: f(x) \leq c + \delta\}$ infinitely often.
\end{lemma}
\begin{proof}
The negation of the statement is
\begin{align*}
\lnot 
\bigl(
\forall \delta >0, \forall n, \exists n_k, (n_k \geq n) \land (z_{n_k} \in A_{\delta})
\bigr)
\end{align*}
which can be rewritten as
\begin{align*}
\exists \delta_0 >0, \exists k_0, \forall k \geq k_0,\  x_{k} \not \in A_{\delta}.
\end{align*}
This means that for all $k \geq k_0$,
\begin{align*}
f(x_k)> c + \delta_0.
\end{align*}
Taking the infimum, we see that  
\begin{align*}
\inf_{k \geq k_0} f(x_k) \geq c + \delta_0.
\end{align*}
Since $\inf_{k \geq k_0} f(x_k)$ is increasing, we have for $k \geq k_0$ that
\begin{align*}
\inf_{m \geq k} f(x_m)  \geq \inf_{m \geq k_0} f(x_m) \geq c + \delta_0.
\end{align*}
Taking the supremum over $k$, we see that we must have
\begin{align*}
\liminf_k f(x_k) = 
\sup_{k \geq 0} \inf_{m \geq k} f(x_m)  \geq c + \delta_0,
\end{align*}
i.e. $\liminf_{k \to \infty} f(x_k) \geq c + \delta_0$, which is a contradiction.
\end{proof}

\begin{lemma}\label{lemma:smoothbound}
  Let $F$ be bounded from below by $F_*$. If $F$ is $(L_0,L_1)-$smooth, it holds that
  \begin{align*}
    \lVert \nabla F(q) \rVert \leq
    2 L_1(F(q) - F_*)  +    \frac{ L_0}{L_1}.
  \end{align*} 
  If $F$ is instead $L$-smooth with Lipschitz constant $L$, it holds that
  \begin{align*}
    \lVert \nabla F(q) \rVert_2^2  \leq 2 L ( F(q) - F_*).
  \end{align*}
\end{lemma}

\begin{proof}
  Consider first the $(L_0,L_1)-$smooth case.
Put
\begin{align}\label{eq:qplus}
q_+ = q - \frac{1}{L_1 \lVert \nabla F(q) \rVert} \nabla F(q),
\end{align}

Then 
\begin{align*}
\lVert q_+ - q \rVert  = \frac{1}{L_1} 
\end{align*}
Thus, the conditions for (21) in \cite{zhang_2020} are satisfied, and it holds that
\begin{align*}
F(q_+) - F(q) \leq \langle \nabla F(q) , q_+ - q\rangle + \frac{ L_0 + L_1\lVert \nabla F(q) \rVert}{2} \lVert q_+ - q \rVert^2.
\end{align*}
Inserting (\ref{eq:qplus}) into the previous expression we see that
\begin{align*}
F(q_+) - F(q) \leq - \frac{1}{L_1 } \lVert \nabla F(q) \rVert + \frac{L_0 + L_1 \lVert \nabla F(q) \rVert}{2} \frac{1}{L_1^2},
\end{align*}
Rearranging the terms, we find that
\begin{align*}
F(q_+) - F(q) \leq - \frac{1}{2L_1 } \lVert \nabla F(q) \rVert + \frac{ L_0}{2} \frac{1}{L_1^2}
\end{align*}
One more rearrangement yields 
\begin{align*}
\lVert \nabla F(q) \rVert \leq
2L_1(F(q) - F(q_+))  +    \frac{L_0}{L_1}.
\end{align*}
Since  $F(q_+) \geq F_*$ we obtain the statement of the first part of the Lemma. The proof of the second part is very similar but simpler, and therefore omitted.
\end{proof}

\begin{lemma}\label{lemma:asboundstochasticgradient}
Let $F$ satisfy Assumption \ref{ass:setting_three}.\ref{ass:empiricalriskminimization} and $f(\cdot,\xi)$ satisfy Assumption \ref{ass:setting_three}.\ref{ass:stochasticfunctionsL0L1smooth}. Then
\begin{align}\label{ineq:asboundstochasticgradient}
\lVert \nabla f(x,\xi) \rVert_2 \leq 2NL_1 \left(F(q) - F_* \right) + \frac{L_0}{L_1},
\end{align}
almost surely,
where $F_* = \frac{1}{N} \sum_{i=1}^N \inf_{q \in \mathbb{R}^d} f_i(x)$.
\end{lemma}
\begin{proof}
We start by showing that there exists a constant $C$ such that
\begin{align}\label{ineq:boundstochasticfunctions}
f(x,\xi) - \inf_{x \in \mathbb{R}^d} f(x,\xi)
\leq C (F(x) - F_*),
\end{align}
almost surely.
First we note that by the properties of $\inf$ it holds that
\begin{align*}
\inf_{x \in \mathbb{R}^d} f(x,\xi) = 
\inf_{x \in \mathbb{R}^d} \biggl( \frac{1}{|B_{\xi}|} \sum_{i \in B_{\xi}} f_i(x)
\biggr)
&= 
 \frac{1}{|B_{\xi}|} \cdot 
\inf_{x \in \mathbb{R}^d} \biggl( \sum_{i \in B_{\xi}} f_i(x) \biggr) \\
&\geq 
 \frac{1}{|B_{\xi}|} \cdot 
\sum_{i \in B_{\xi}} \inf_{x \in \mathbb{R}^d}  f_i(x).
\end{align*}
Hence 
\begin{align*}
\begin{split}
f(x,\xi) - \inf_{x \in \mathbb{R}^d} f(x,\xi)
&=
 \frac{1}{|B_{\xi}|} \cdot 
 \sum_{i \in B_{\xi}} f_i(x) 
 - \inf_{x \in \mathbb{R}^d} f(x,\xi)
 \\ 
 & \leq
  \frac{1}{|B_{\xi}|} \cdot 
 \sum_{i \in B_{\xi}} f_i(x) 
 -
  \frac{1}{|B_{\xi}|} \cdot 
\sum_{i \in B_{\xi}} \inf_{x \in \mathbb{R}^d}  f_i(x)
\end{split}
\end{align*}
Since $f_i(x) - \inf_{x \in \mathbb{R}^d}  f_i(x)) \geq 0$, the previous expression can be bounded by
\begin{align*}
  \frac{1}{|B_{\xi}|} \cdot 
 \sum_{i=1}^N f_i(x) 
 - \inf_{x \in \mathbb{R}^d}  f_i(x)
=
\frac{N}{|B_{\xi}|} \left( F(x) - F_*\right).
\end{align*}
As $1 \le |B_{\xi}| \le N$,
(\ref{ineq:boundstochasticfunctions}) holds.
Since $f(\cdot,\xi)$ is $(L_0,L_1)$-smooth,
it further holds that
\begin{align*}
\lVert \nabla f(x,\xi) \rVert_2
\leq
2 L_1 \left(
f(x,\xi) - \inf_{x \in \mathbb{R}^d} f(x,\xi)
\right) + \frac{L_0}{L_1}.
\end{align*}
Combining the previous expression with (\ref{ineq:boundstochasticfunctions}), we obtain (\ref{ineq:asboundstochasticgradient}).
\end{proof}

\begin{remark}
Note that for fixed, \emph{deterministic} $x$, (\ref{ineq:asboundstochasticgradient}) implies that the norm $\lVert \nabla f(x,\xi) - \nabla F(x) \rVert_2$ is bounded almost surely. This means that around a stationary point $q_*$ or for the initial iterate $q_0$ (which in this paper is assumed to be deterministic), the norm of the noise is not heavy-tailed. When $x= q_k$ is a random variable this is no longer the case. This is in line with e.g \citep{gurbuzbalaban21_2021}, in which it is reported that the noise is not heavy-tailed initially.
\end{remark}

\begin{lemma}\label{lemma:2ndmomentsetting3}
Under Assumptions \ref{ass:setting_three}. \ref{ass:empiricalriskminimization}-\ref{ass:stochasticfunctionsL0L1smooth} it holds that
\begin{align}\label{eq:2ndmomentsetting3}
\mathbb{E} \left[ \lVert \nabla F(q_k) - \nabla f(q_k,\xi_k) \rVert_2^2 \right] < + \infty,
\end{align}
where $\{q_k\}_{k \geq 0}$ is determined by \eqref{eq:stochastic_Hamiltonian_descent}. 
\end{lemma}
\begin{proof}
By Lemma \ref{lemma:asboundstochasticgradient} and the triangle inequality it holds that
\begin{align*}
 \lVert \nabla F(q_k) - \nabla f(q_k,\xi_k) \rVert_2
 &\leq
  2NL_1 \left(F(q_k) - F_* \right) + \frac{L_0}{L_1}
  + \lVert \nabla F(q_k) \rVert_2
  \\
  & \leq 
   2L_1(N+1) \left(F(q_k) - F_* \right) + \frac{2L_0}{L_1},
\end{align*}
where we have used Lemma \ref{lemma:smoothbound} to bound the $\lVert \nabla F(q_k)\rVert_2-$term in the second inequality.
We can use this to bound one of the powers in the expectation of \eqref{eq:2ndmomentsetting3}:
\begin{align*}
\mathbb{E}_{\xi_k} \left[ \lVert \nabla F(q_k) - \nabla f(q_k,\xi_k) \rVert_2^2 \right]
&\leq
\mathbb{E}_{\xi_k} \left[ \lVert \nabla F(q_k) - \nabla f(q_k,\xi_k) \rVert_2  \right]  \\
&\cdot \left(  2L_1(N+1) \left(F(q_k) - F_* \right) + \frac{2L_0}{L_1} \right)\\
& \leq \sigma \cdot \left( 2L_1(N+1) \left(F(q_k) - F_* \right) + \frac{2L_0}{L_1} \right),
\end{align*}
where we have used Assumption \ref{ass:setting_three}.\ref{ass:heavytailednoise}, along with the fact that $q_k$ is $\mathcal{F}_k$-measurable (hence we may take the $F(q_k)$-term outside of the conditional expectation) while $\xi_k$ is independent of $\mathcal{F}_k$. Taking the total expectation of this, we find that
\begin{align*}
\mathbb{E} \left[ \lVert \nabla F(q_k) - \nabla f(q_k,\xi_k) \rVert_2^2 \right]
\leq \sigma \cdot \left( 2L_1(N+1) \mathbb{E} \left[F(q_k) - F_* \right] + \frac{2L_0}{L_1} \right).
\end{align*}
By Lemma \ref{thm:pqbounded} we have that $\sup_{k \geq 0}  \mathbb{E} \left[F(q_k)  \right] < \infty$, and thus \eqref{eq:2ndmomentsetting3} follows.  
\end{proof}

\section{Existence and uniqueness of solutions}\label{appendix:existence}
When Assumption
\ref{ass:setting_one}.\ref{ass:Lipschitzcontinuous_gradient} is satisfied, the solutions to \eqref{eq:nearly_Hamiltonian_system} are clearly unique and exists for all $t \geq 0$. Under Assumption \ref{ass:setting_two}.\ref{ass:L0L1smoothness}, this is not so obvious anymore. Essentially this follows from Assumption \ref{ass:basic}.\ref{ass:coercive} and \ref{ass:kinetic_energy_function}.\ref{ass:coercive_kinetic}, and in this section we elaborate on why this is the case.
For brevity, let $z= (p,q)$ and let 
\begin{align}\label{eq:nearly_Hamiltonian_system_short}
\dot{z} = f(z)
\end{align}
denote the system \eqref{eq:nearly_Hamiltonian_system}.
Put $V(z) = H(p,q) - F_* - \varphi_*$, where
$F_*$ is defined by \ref{ass:basic}.\ref{ass:proper} and $\varphi_*$ by \ref{ass:kinetic_energy_function}.\ref{ass:proper_kinetic}. Then $V$ is a Lyapunov function for \eqref{eq:nearly_Hamiltonian_system_short} and by construction, $V(z) \geq 0$.
We also note that by Assumption \ref{ass:basic}.\ref{ass:coercive} and \ref{ass:kinetic_energy_function}.\ref{ass:coercive_kinetic}
\begin{align*}
\lim_{\lVert z \rVert \to \infty} V(z) = \infty.
\end{align*}

By Theorem 10.1 in \cite{Yoshizawa_1966} (see also Theorem 1 in \cite{Yoshizawa_1959}), it follows that the solutions to \eqref{eq:nearly_Hamiltonian_system_short} are \emph{equi-bounded}, in the sense that for a given $B>0$, there exists a $C$ such that all solutions with initial conditions $\lVert z_0\rVert \leq B$ satisfies $\lVert z(t) \rVert \leq C$. Thus the solutions to \eqref{eq:nearly_Hamiltonian_system_short} are defined for all $t \geq 0$, compare \S 23 in  \cite{LaSalle_Lefschetz_1961}.

We will now make use of the following theorem from \cite{Yoshizawa_1966}:
\begin{theorem}\label{yoshikawathm}[Thm. 1.4 in \cite{Yoshizawa_1966}]
Suppose that the right hand side of \eqref{eq:nearly_Hamiltonian_system_short} is continuous on $D = \{(t,x): t \in [0,\infty), \lVert x\rVert \leq b \}$. In order that every solution of \eqref{eq:nearly_Hamiltonian_system_short} that starts at a point $(t_0,x_0) \in D$ is unique to the right, it is necessary and sufficient that there exists a neighbourhood $U$ of $(t_0,x_0)$ with the following property: Let $W$ be a set of $(t,x,y)$ such that $(t,x) \in U$ and $(t,y) \in U$. Then there exists a Lyapunov function $V$ defined on $W$ which satisfies
\begin{enumerate}[label=\roman*),start=1]
    \item $V(t,x,y) = 0$ if $x=y$, \label{yoshikawathm1}
    \item $V(t,x,y) > 0$ if $x\neq y$,  \label{yoshikawathm2}
    \item $V$ is locally Lipschitz continuous,  \label{yoshikawathm3}
    \item $\frac{\mathrm{d}}{\mathrm{d} t} V(t,x(t),y(t)) \leq 0$.  \label{yoshikawathm4}
\end{enumerate}
\end{theorem}
From the discussion that precedes Theorem \ref{yoshikawathm}, we can consider $D= \{(t,x): t \in [0, \infty), \lVert x \rVert \leq C\}$, since the solutions to \eqref{eq:nearly_Hamiltonian_system_short} are bounded.
Let $x \in \mathbb{R}^n$. Since $F$ is $(L_0,L_1)$-smooth, we can take the neighbourhood $U$ in the previous theorem to be $U = \{y: \lVert x- y \rVert \leq \frac{1}{L_1}\}$. By $(L_0,L_1)$-smoothness, it holds that $\nabla F$ is continuous. This, along with the fact that $U$ is compact, implies that $\sup_{x\in U} \lVert \nabla F(x)\rVert < +\infty$. Thus, $\nabla F$ is Lipschitz continuous on $U$ with Lipschitz constant $L_0 + L_1 \sup_{x \in U}\lVert \nabla F(x) \rVert$.
As $\nabla \varphi$ is also Lipschitz continuous, so is $f$. Let $L$ be the Lipschitz constant of $f$ on $U$. 
Then we can take
\begin{align*}
V(t,x,y) = e^{-2Lt} \lVert x - y \Vert^2,
\end{align*}
which clearly satifies  \ref{yoshikawathm1}-  \ref{yoshikawathm3} of Theorem \ref{yoshikawathm}. By the calculation
\begin{align*}
\frac{\mathrm{d}}{\mathrm{d} t} V(t,x(t),y(t)) & \leq
2 e^{-2Lt} \left(
-L  \lVert x(t) - y(t) \Vert^2
+
 \lVert x(t) - y(t) \Vert
  \lVert f(t) - f(t) \Vert\right)
  \\ & \leq 0,
\end{align*}
 it also follows that it satisfies \ref{yoshikawathm4}.
Hence, we can appeal to Theorem \ref{yoshikawathm} to conclude that the solutions to \eqref{eq:nearly_Hamiltonian_system_short} are unique for $t \geq0$.

\section{The discontinuous Arzel\`a-Ascoli theorem}\label{appendix:arzela_ascoli}
The particular version of the Arzel\`a-Ascoli theorem we make use of is an extension of that in \cite{droniou_eymard_2016} to processes on the entire real line. It is mentioned without proof in \cite{kushner_yin_2003}. For the sake of completeness we include it here. The issue with applying Theorem 6.2 in \cite{droniou_eymard_2016} is that without further knowledge, the subsequences (and the limits we extract) may differ on different compact intervals, since these are different spaces. Here we show that this is not the case. 

Recall that a function $f: \mathbb{R}^+ \to \mathbb{R}^d$ is c{\`a}dl{\`a}g if the left limit $f(t_-) = \lim_{s \to t^-}f(s)$ exists and the right limit
$f(t_+) = \lim_{s \to t^+}f(s)$ exists and equals $f(t)$, compare \cite{billingsley1968}. All processes considered in this paper are c{\`a}dl{\`a}g by virtue of being piecewise constant.
For each $m \in \mathbb{N}$ we let $D_m$ be the space of c{\`a}dl{\`a}g functions on the interval $[0,m]$ endowed with the uniform metric  $\lVert x \rVert_m= \sup_{t \in [0,m]} |x(t)|$, and $D_{\infty}$ be the space of c{\`a}dl{\`a}g functions on $\mathbb{R}_+$ endowed with the topology of compact convergence. Furthermore, we introduce for each $m$ the restriction mapping $\psi_m: D_{\infty} \to D_m$, given by $\psi_m(x)(t) = x(t)$, $t \in [0,m]$.
We have

\begin{lemma}\label{lemma:relativelycompact}
A set $A$ is relatively compact in $D_{\infty}$ if and only if for each $m$, $\psi_m(A)$ is relatively compact in $D_m$. 
\end{lemma}
The proof is very similar to that of Theorem 16.4 in \cite{billingsley1968} and therefore omitted.

\begin{theorem}[Arzel\`a-Ascoli]
Let $\{x_k\}_{k \geq 0}\subset D_{\infty}$ be a sequence of functions that are equicontinuous in the extended sense. Then there exists a subsequence $\{x_{n_k}\}_{k \geq 0} \subset \{x_k\}_{k \geq 0}$ that converges uniformly on compact sets to a an element $x \in C([0, \infty))$.
\end{theorem}
\begin{proof}
Take $A = \{x_k\}_{k \geq 0}$. By Theorem 6.2 in \cite{droniou_eymard_2016}, the restriction $\psi_m(A)$ is relatively compact for each $m \in \mathbb{N}$. By Lemma \ref{lemma:relativelycompact}, $A$ is relatively compact in $D_{\infty}$, and thus there exists a subsequence $\{x_{n_k}\}_{k \geq 0} \subset \{x_k\}_{k \geq 0}$ that converges to some element $x \in D_{\infty}$. It remains to show that $x \in C([0, \infty))$.

Since $x_{n_k}$ converges to $x \in D_{\infty}$ it converges to $x$
on each compact interval $[0,m]$. Furthermore, $\{x_{n_k}\}_{k \geq 0}$ is equicontinous in the extended sense since it is a subsequence of $\{x_k\}_{k \geq 0}$. Hence, on each interval $[0,m]$ it has a further subsequence which converges uniformly to a continuous function. This limit must coincide with the restriction of $x$ to $[0,m]$. Thus $x$ is continuous on each interval $[0,m]$ and so $x \in C([0, \infty))$.
\end{proof}

\end{document}